\documentclass{article}

\usepackage{bbm}
\usepackage{mathtools}
\usepackage[numbers]{natbib}
\usepackage{verbatim}
\usepackage{comment}
\usepackage[linesnumbered, ruled,vlined]{algorithm2e}
\usepackage{algorithm,algorithmic}

\usepackage{microtype}
\usepackage{graphicx}
\usepackage{subfigure}
\usepackage{booktabs} 
\usepackage{adjustbox}
\usepackage{hyperref}

\usepackage[preprint]{nips_2018}

\usepackage{tikz}
\usetikzlibrary{intersections,shapes.arrows}
\newcommand\Dist[1]{\phantom{\rule{#1}{4pt}}}

\usepackage{amsthm,amsmath,amssymb, amsfonts}
\newtheorem{theorem}{Theorem}[section]
\newtheorem{lemma}[theorem]{Lemma}

\usepackage{comment}
\usepackage{mathrsfs}
\usepackage{enumitem}

\title{From Complexity to Simplicity: Adaptive ES-Active Subspaces for Blackbox Optimization}

\author{
   Krzysztof Choromanski\textsuperscript{*} \\
  Google Brain Robotics \\
  kchoro@google.com \\
  \And
   Aldo Pacchiano\textsuperscript{*} \\
  UC Berkeley \\
 pacchiano@berkeley.edu \\
  \And
  Jack Parker-Holder\textsuperscript{*} \\
  Columbia University \\
 jh3764@columbia.edu \\
  \And
   Yunhao Tang\textsuperscript{*} \\
  Columbia University \\
  yt2541@columbia.edu \\
}

\begin{document}

\maketitle

\begin{abstract}
We present a new algorithm ($\mathrm{ASEBO}$) for optimizing high-dimensional blackbox functions. $\mathrm{ASEBO}$ 
adapts to the geometry of the function and learns optimal sets of sensing directions, which are used to probe it, on-the-fly. 
It addresses the exploration-exploitation trade-off of blackbox optimization with expensive blackbox queries by continuously learning the bias of the lower-dimensional model 
used to approximate gradients of smoothings of the function via compressed sensing and contextual bandits methods. To obtain this model, it leverages techniques from the emerging theory of active subspaces \cite{constantine} in the novel ES blackbox optimization context.
As a result, $\mathrm{ASEBO}$ learns the dynamically changing intrinsic dimensionality
of the gradient space and adapts to the hardness of different stages of the optimization without external supervision. Consequently, it leads to more sample-efficient blackbox optimization than state-of-the-art algorithms.
We provide theoretical results and test $\mathrm{ASEBO}$ advantages over other methods empirically by evaluating it on the set
of reinforcement learning policy optimization tasks as well as functions from the recently open-sourced $\mathrm{Nevergrad}$ library. 
\end{abstract}

\section{Introduction}\label{sec:intro}

Consider a high-dimensional function $F:\mathbb{R}^{d} \rightarrow \mathbb{R}$. We assume that querying it is expensive.
Examples include reinforcement learning (RL) blackbox functions taking as inputs vectors $\theta$ encoding policies 
$\pi: \mathcal{S} \rightarrow \mathcal{A}$ mapping states to actions and outputting total (expected/discounted) rewards
obtained by agents applying $\pi$ in given environments \cite{Gym}. For this class of functions evaluations usually require running a simulator. Other examples include wind configuration design optimization problems for high speed civil transport aircrafts,
optimizing computer codes (e.g. NASA synthetic tool $\mathrm{FLOPS/ENGENN}$ used to size the aircraft and propulsion system \cite{nasa}), 
crash tests, medical and chemical reaction experiments \cite{zhou}.

Evolution strategy (ES) methods have traditionally been used in low-dimensional regimes (e.g. hyperparameter tuning), and considered ill-equipped for higher dimensional problems due to poor sampling complexity \cite{nesterov}. However, a flurry of recent work has shown they can scale better than previously believed \cite{petroski, conti, montreal, horia, stockholm, ES, TRES}. This is thanks to a couple of reasons.


First of all, new ES methods apply several efficient heuristics (filtering, various normalization techniques as in \cite{horia} and new exploration strategies as in \cite{conti}) in order to substantially 
improve sampling complexity. Other recent methods \cite{montreal, stockholm} are based on 
more accurate Quasi Monte Carlo (MC) estimators of the gradients of Gaussian smoothings of blackbox functions with theoretical guarantees. These approaches
provide better quality gradient sensing mechanisms.
Additionally, in applications such as RL, new compact structured policy architectures 
(such as low-displacement rank neural networks from \cite{stockholm} or even linear policies \cite{worldmodels}) are used to
reduce the number of policies' parameters and dimensionality of the optimization problem. 

Recent research also shows that ES-type blackbox optimization in RL leads to more stable policies than policy gradient methods since ES methods search for parameters that are robust to perturbations \cite{lehman}. Unlike policy gradient methods, ES aims to find parameters maximizing expected reward (rather than just a reward) in respect to Gaussian perturbations.

Finally, pure ES methods as opposed to state-of-the-art policy optimization techniques ($\mathrm{TRPO}$, $\mathrm{PPO}$ or $\mathrm{ARS}$  \cite{schulman2017proximal, babadi, trpo, horia}), can be applied also for blackbox optimization problems 
that do not exhibit $\mathrm{MDP}$ structure required for policy gradient methods 
and cannot benefit from state normalization algorithm central to $\mathrm{ARS}$. This has led to their recent popularity for non-differentiable tasks  \cite{EPG, worldmodels}.



In this paper we introduce a new adaptive sample-efficient blackbox optimization algorithm. $\mathrm{ASEBO}$ adapts to the geometry of blackbox functions and learns optimal sets of sensing directions, which are used to probe them, on-the-fly. To do this, it leverages techniques from the emerging theory of active subspaces \cite{constantine, constantine_2, constantine_3, liliu} in a novel ES blackbox optimization context.
Active subspaces and their extensions are becoming popular as effective techniques for dimensionality reduction (see for instance: active manifolds \cite{bridges} or ResNets for learning isosurfaces \cite{resnetzhang}). However, to the best of our knowledge we are the first to apply active subspace ideas for ES optimization.

$\mathrm{ASEBO}$ addresses the exploration-exploitation trade-off of blackbox optimization with expensive function queries by continuously learning the bias of the lower-dimensional model used to approximate gradients of smoothings of the function with compressed sensing and contextual bandits methods. 
The adaptiveness is what distinguishes it from some recently introduced guided ES methods such as \cite{metz} that rely on 
fixed hyperparameters that are hard to tune in advance (e.g. the length of the buffer defining lower dimensional space for gradient search). 
We provide theoretical results and empirically evaluate $\mathrm{ASEBO}$ on a set of RL blackbox optimization tasks as well as non-RL blackbox functions from the recently open-sourced $\mathrm{Nevergrad}$ library
\cite{nevergrad}, showing that it consistently learns optimal inputs with fewer queries to a blackbox function than other methods. 

\textbf{$\mathrm{ASEBO}$ versus CMA-ES:} 
There have been a variety of works seeking to reduce sampling complexity for ES methods through the use of metric learning. 
The prominent class of the covariance matrix adaptation evolution strategy (CMA-ES) methods derives state-of-the-art derivative free blackbox optimization algorithms, which seek to learn 
and maintain a fully parameterized Gaussian distribution. 
CMA-ES suffers from quadratic time complexity for each evaluation which can be limiting for high dimensional problems. As such, a series of attempts have been made to produce scalable 
variants of CMA-ES, by restricting the covariance matrix to the diagonal (sep-CMA-ES \cite{cmaes_mod}) or a low rank approximation as in VD-CMA-ES \cite{vdcmaes} and LM-CMA-ES \cite{lmcmaes}. 
Two recent algorithms, VkD-CMA-ES \cite{vkdcmaes} and LM-MA-ES \cite{lmmaes}, seek to combine the above ideas and have been shown to be successful in large-scale settings, 
including RL policy learning \cite{challenges_esrl}. Although these methods are able to quickly learn and adapt the covariance matrix, they are heavily dependent on hyperparameter 
selection \cite{vkdcmaes, cmaes_compare} and lack the means to avoid learning a bias. As our experiments show, this can severely hurt their performance. The best CMA-ES variants often struggle with RL tasks of challenging objecive landscapes, displaying inconsistent performance across tasks. Furthermore,
they require careful hyperparameter tuning for good performance (see: analysis in Section \ref{sec:experiments_overleaf}, Fig. \ref{fig:lmmaes}).



\vspace{-3mm} 
\section{Adaptive Sample-Efficient Blackbox Optimization}\label{sec:algorithm}

Before we describe $\mathrm{ASEBO}$, we explain key theoretical ideas behind the algorithm.
$\mathrm{ASEBO}$ uses online $\mathrm{PCA}$ to maintain and update on-the-fly subspaces which we call 
\textit{ES-active subspaces} $\mathcal{L}^{\mathrm{ES}}_{\mathrm{active}}$, accurately approximating the gradient data space at 
a given phase of the algorithm. The bias of the obtained gradient estimators is measured by sensing the length 
of its component from the orthogonal complement $\mathcal{L}_{\mathrm{active}}^{\mathrm{ES}, \perp}$ via compressed sensing 
or computing optimal probabilities for exploration (e.g. sensing from $\mathcal{L}_{\mathrm{active}}^{\mathrm{ES}, \perp}$)
via contextual bandits methods \cite{shipra}. The algorithm corrects its probabilistic distributions used for choosing directions for gradient sensing
based on these measurements.
As we show, we can measure that bias accurately using only a constant number of additional function queries, regardless of the dimensionality. 
This in turn determines an exploration strategy, as we explain later. Estimated gradients are then used to update parameters.

\subsection{Preliminaries}

Consider a blackbox function $F:\mathbb{R}^{d} \rightarrow \mathbb{R}$. We do not assume that $F$ is differentiable.
The \textit{Gaussian smoothing} \cite{nesterov} $F_{\sigma}$ of $F$ parameterized by smoothing parameter $\sigma>0$ is given as:
$
F_{\sigma}(\theta) = \mathbb{E}_{\mathbf{g} \in \mathcal{N}(0,\mathbf{I}_{d})}[F(\theta + \sigma \mathbf{g})] = 
(2\pi)^{-\frac{d}{2}}
\int_{\mathbb{R}^{d}}F(\theta + \sigma \mathbf{g})e^{-\frac{\|\mathbf{g}\|^{2}}{2}}d\mathbf{g}.
$
The gradient of the Gaussian smoothing of $F$ is given by the formula:
\vspace{-2mm}
\begin{equation}
\label{grad}
\nabla F_{\sigma}(\theta)=\frac{1}{\sigma}\mathbb{E}_{\mathbf{g} \sim \mathcal{N}(0,\mathbf{I}_{d})}[F(\theta + \sigma \mathbf{g})\mathbf{g}].    
\end{equation}
\vspace{-3mm}

Formula \ref{grad} on $\nabla F_{\sigma}(\theta)$ leads straightforwardly to several unbiased Monte Carlo (MC) estimators of $\nabla F_{\sigma}(\theta)$, where the most
popular ones are: the \textit{forward finite difference} estimator \cite{stockholm} defined as:
$
\widehat{\nabla}^{\mathrm{FD}}_{\mathrm{MC}}F_{\sigma}(\theta) = 
\frac{1}{k\sigma}\sum_{i=1}^{k}(F(\theta + \sigma \mathbf{g}_{i})-F(\theta))\mathbf{g}_{i},
$
and an \textit{antithetic ES gradient estimator} \cite{ES} given as:
$
\widehat{\nabla}^{\mathrm{AT}}_{\mathrm{MC}}F_{\sigma}(\theta) = 
\frac{1}{2k\sigma}\sum_{i=1}^{k}(F(\theta + \sigma \mathbf{g}_{i})-F(\theta - \sigma \mathbf{g}_{i}))\mathbf{g}_{i},
$
where typically $\mathbf{g}_{1},...,\mathbf{g}_{k}$ are taken independently at random from $\mathcal{N}(0,\mathbf{I}_{d})$ of from more complex joint distributions for variance reduction (see: \cite{stockholm}). We call samples $\mathbf{g}_{1},...,\mathbf{g}_{k}$ the \textit{sensing directions} since they are used to sense gradients $\nabla F_{\sigma}(\theta)$.
The antithetic formula can be alternatively rationalized as giving the renormalized gradient of $F$ (if $F$ is smooth), if not taking into account cubic and higher-order terms
of the Taylor expansion $F(\theta + \mathbf{v}) = F(\theta) + \nabla F^{\top} \mathbf{v} + \frac{1}{2}\mathbf{v}^{\top}H(\theta)\mathbf{v}$
(where $H(\theta)$ stands for the Hessian of $F$ in $\theta$).

Standard ES methods apply different gradient-based techniques such as $\mathrm{SGD}$ or $\mathrm{Adam}$, 
fed with the above MC estimators of $\nabla F_{\sigma}$ to conduct blackbox optimization. The number of samples $k$ per iteration
of the optimization procedure is usually of the order $O(d)$. This becomes a computational bottleneck for high-dimensional blackbox
functions $F$ (for instance, even for relatively small RL tasks with policies encoded by compact neural networks we still have $d>100$ parameters).

\begin{figure}[!t]
\begin{minipage}{0.99\textwidth}
  \subfigure[\textbf{HC:} active subspace]{\includegraphics[keepaspectratio, width=0.24\textwidth]{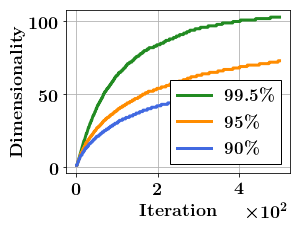}}  
  \subfigure[\textbf{SW:} active subspace]{\includegraphics[keepaspectratio, width=0.22\textwidth]{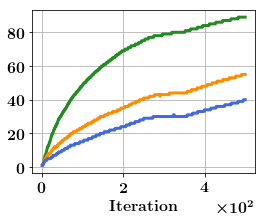}}
  \subfigure[\textbf{HC:} \# of samples]{\includegraphics[keepaspectratio, width=0.24\textwidth]{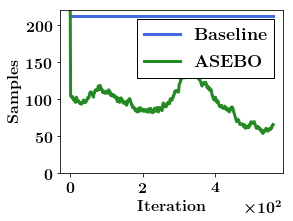}}
  \subfigure[\textbf{SW:} \# of samples]{\includegraphics[keepaspectratio, width=0.22\textwidth]{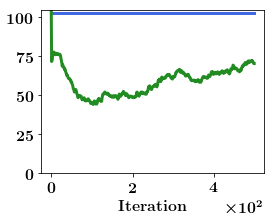}}
  \end{minipage}
  \caption{The motivation behind $\mathrm{ASEBO}$. Two first plots: ES baseline for $\mathrm{HalfCheetah}$ and $\mathrm{Swimmer}$ tasks from the $\mathrm{OpenAI}$ $\mathrm{Gym}$ library for 
  $212$-dimensional policies - the plot shows how the dimensionality of the space capturing a given percentage of variance of approximate gradient data depends on the iteration of the algorithm. This information is never exploited by the algorithm, even though $99.5\%$ of the variance resides in the much lower-dimensional space ($100$ dimensions).  Two last plots: $\mathrm{ASEBO}$ taking advantage of this information (\# of sample/sensing directions reflects the hardness of the optimization at each iteration and is strongly correlated with the $\mathrm{PCA}$ dimensionality. 
  }
  \label{fig:pca_motivation} 
\vspace{-3mm}  
\end{figure}
\vspace{-3mm}  

\subsection{ES-active subspaces via online PCA with decaying weights}

The first idea leading to the $\mathrm{ASEBO}$ algorithm is that in practice one does not need to estimate the gradient of $F$ accurately 
(after all ES-type methods do not even aim to compute the gradient of $F$, but rather focus on $\nabla F_{\sigma}$). 
Poor scalability of 
ES-type blackbox optimization algorithms is caused by high-dimensionality of the gradient vector. However, during the optimization process the 
space spanned by gradients may be locally well approximated by a lower-dimensional subspace $\mathcal{L}$ and sensing the gradient in that subspace might be more effective. 
In some recent papers such as \cite{metz} such a subspace is defined simply as 
$\mathcal{L} = \mathrm{span}\{\widehat\nabla^{\mathrm{AT}}_{\mathrm{MC}}F_{\sigma}(\theta_{i}),\widehat\nabla^{\mathrm{AT}}_{\mathrm{MC}}F_{\sigma}(\theta_{i-1}),...,\widehat\nabla^{\mathrm{AT}}_{\mathrm{MC}}F_{\sigma}(\theta_{i-s+1})\}$, where $\{\widehat\nabla^{\mathrm{AT}}_{\mathrm{MC}}F_{\sigma}(\theta_{i}),\widehat\nabla^{\mathrm{AT}}_{\mathrm{MC}}F_{\sigma}(\theta_{i-1}),...,\widehat\nabla^{\mathrm{AT}}_{\mathrm{MC}}F_{\sigma}(\theta_{i-s+1})\}$ 
stands for the batch of last $s$ approximated gradients during the optimization process and $s$ is a fixed hyperparameter.
Even though $\mathcal{L}$ will dynamically change during the optimization, such an approach has several disadvantages in practice. 
Tuning parameter $s$ is very difficult or almost impossible and the assumption that the dimensionality of $\mathcal{L}$ should be constant during 
optimization is usually false. In our approach, dimensionality of $\mathcal{L}$ varies and depends on the hardness of the optimization in different optimization stages.

We apply Principal Component Analysis ($\mathrm{PCA}$, \cite{Jolliff}) to create a subspace $\mathcal{L}$ capturing particular variance $\sigma > 0$ of 
the approximate gradients data. 
This data is either: the approximate gradients computed in previous iterations from the antithetic formula or: the elements of the sum from that equation that are averaged over to obtain these gradients. 
For clarity of the exposition, from now on we will assume the former, but both variants are valid.  
Choosing $\sigma$ is in practice much easier than $s$ and leads to 
subspaces $\mathcal{L}$ of varying dimensionalities throughout the optimization procedure, called by us from now on \textit{ES-active subspaces} 
$\mathcal{L}_{\mathrm{active}}^{\mathrm{ES}}$. 
\vspace{-3mm}  
\begin{algorithm}[H]
\caption{$\mathrm{ASEBO}$ Algorithm}
\textbf{Hyperparameters:}  number of iterations of full sampling $l$, smoothing parameter $\sigma>0$,
step size $\eta$, PCA threshold $\epsilon$, decay rate $\gamma$, 
total number of iterations $T$.\; \\
\textbf{Input:}  blackbox function $F$, vector $\theta_0 \in \mathbb{R}^{d}$ where optimization starts. 
                 $\mathrm{Cov}_0 \in \{0\}^{d \times d}$, $p^{0}=0$.\; \\
\textbf{Output:} vector $\theta_{T}$. \; \\
\For{$t=0, \ldots, T-1$}{
  \If {$t < l$}{ 
    Take $n_t = d$. Sample $\mathbf{g}_1, \cdots, \mathbf{g}_{n_{t}}$ from $\mathcal{N}(0, \mathbf{I}_{d})$ (independently). \; 
  }
  \Else {
    1. Take top $r$ eigenvalues $\lambda_{i}$ of $\mathrm{Cov}_{t}$, where $r$ is smallest such that: $\sum_{i=1}^{r}\lambda_{i} \geq \epsilon \sum_{i=1}^{d} \lambda_{i}$, 
       using its $\mathrm{SVD}$ as described in text and take $n_{t}=r$.\; \\
    2. Take the corresponding eigenvectors $\mathbf{u}_{1},...,\mathbf{u}_{r} \in \mathbb{R}^{d}$ and let $\mathbf{U} \in \mathbb{R}^{d \times r}$
       be obtained by stacking them together. Let $\mathbf{U}^{\mathrm{act}} \in \mathbb{R}^{d \times r}$ be obtained from stacking together 
       some orthonormal basis of $\mathcal{L}^{\mathrm{ES}}_{\mathrm{active}}   
       \overset{\mathrm{def}}{=} \mathrm{span}\{\mathbf{u}_{1},...,\mathbf{u}_{r}\}$.
       Let $\mathbf{U}^{\perp} \in \mathbb{R}^{d \times (d-r)}$ be obtained from
       stacking together some orthonormal basis of the orthogonal complement $\mathcal{L}^{\mathrm{ES}, \perp}_{\mathrm{active}}$ of $\mathcal{L}^{\mathrm{ES}}_{\mathrm{active}}$.       
       \\
    3. Sample $n_{t}$ vectors $\mathbf{g}_{1},...,\mathbf{g}_{n_{t}}$ as follows: with probability $1-p^{t}$ from $\mathcal{N}(0,\mathbf{U}^{\perp}(\mathbf{U}^{\perp})^{\top})$
       and with probability $p^{t}$ from  $\mathcal{N}(0,\mathbf{U}^{\mathrm{act}}(\mathbf{U}^{\mathrm{act}})^{\top})$. \\
    4. Renormalize $\mathbf{g}_{1},...,\mathbf{g}_{n_{t}}$ such that marginal distributions $\|\mathbf{g}_{i}\|_{2}$ are  $\chi(d)$.
  }
  1. Compute $\widehat{\nabla}_{\mathrm{MC}}^{\mathrm{AT}}F(\theta_{t})$ as:
    $
    \widehat{\nabla}_{\mathrm{MC}}^{\mathrm{AT}}F(\theta_{t}) =  \frac{1}{2n_{t}\sigma} \sum_{j=1}^{n_{t}} (F(\theta_t +  \mathbf{g}_j) - F(\theta_t - \mathbf{g}_j)) \mathbf{g}_j. \\
    $
  2. Set $\mathrm{Cov}_{t+1} = \lambda \mathrm{Cov}_{t} + (1-\lambda) \Gamma$, where 
     $\Gamma = \widehat\nabla^{\mathrm{AT}}_{\mathrm{MC}}F_{\sigma}(\theta_{t}) (\widehat\nabla^{\mathrm{AT}}_{\mathrm{MC}}F_{\sigma}(\theta_{t}))^{\top}$.
     \\
  3. Set $p^{t+1} = p_{\mathrm{opt}}$ for $p_{\mathrm{opt}}$ output by Algorithm 2 and: $\theta_{t+1} \leftarrow \theta_{t} + \eta \widehat{\nabla}_{\mathrm{MC}}^{\mathrm{AT}}F(\theta_{t})$. 
 }
\label{Alg:asebo}
\end{algorithm}
\vspace{-3mm}  

These will be in turn applied to define covariance matrices encoding  probabilistic distributions applied to construct sensing directions used for estimating $\nabla F_{\sigma}(\theta)$. 
The additional advantage of our approach is that $\mathrm{PCA}$ automatically filters out gradient noise.

We use our own online version of $\mathrm{PCA}$ with decaying weights (decay rate is defined by parameter $\lambda>0$). By tuning $\lambda$ we can define the rate at which 
historical approximate gradient data is used to choose the right sensing directions, which will continuously decay. 
We consider a stream of approximate gradients $\widehat\nabla^{\mathrm{AT}}_{\mathrm{MC}}F_{\sigma}(\theta_{0}),...\widehat\nabla^{\mathrm{AT}}_{\mathrm{MC}}F_{\sigma}(\theta_{i}),...$ 
obtained during the optimization procedure. We maintain and update on-the-fly the covariance matrix $\mathrm{Cov}_{t}$, where $t$ stands for the number of completed iterations, 
in the form of the symmetric matrix $\mathrm{SVD}$-decomposition 
$\mathrm{Cov}_{t} = \mathbf{Q}_{t}^{\top}\Sigma_{t}\mathbf{Q}_{t} \in \mathbf{R}^{d}$. 
When the new approximate gradient  $\widehat\nabla^{\mathrm{AT}}_{\mathrm{MC}}F_{\sigma}(\theta_{t})$ arrives, the update of the covariance matrix is driven by the following equation, 
reflecting data decay process, where $\mathbf{x}_{t} = \widehat\nabla^{\mathrm{AT}}_{\mathrm{MC}}F_{\sigma}(\theta_{t})$:
\begin{equation}
\label{cov}
\mathbf{Q}_{t+1}^{\top}\Sigma_{t+1}\mathbf{Q}_{t+1} = 
\lambda \mathbf{Q}_{t}^{\top}\Sigma_{t}\mathbf{Q}_{t} + (1-\lambda)
\mathbf{x}_{t} \mathbf{x}_{t} ^{\top},
\end{equation}

To conduct the update cheaply, it suffices to observe that the RHS of Equation \ref{cov}
can be rewritten as: $\lambda \mathbf{Q}_{t}^{\top}\Sigma_{t}\mathbf{Q}_{t} + (1-\lambda)
\mathbf{x}_{t} \mathbf{x}_{t} ^{\top} = 
\mathbf{Q}_{t}^{\top}(\lambda \Sigma_{t} + (1-\lambda)\mathbf{Q}_{t}\mathbf{x}_{t}
(\mathbf{Q}_{t}\mathbf{x}_{t})^{\top})\mathbf{Q}_{t}$.
Now, using the fact that for a matrix of the form $\mathbf{D}+\mathbf{u}\mathbf{u}^{\top}$,
we can get its decomposition in time $O(d^{2})$ \citep{golub}, we obtain an algorithm performing updates in quadratic time. That in practice 
suffices since the bottleneck of the computations is in querying $F$ and additional overhead related to updating $\mathcal{L}^{\mathrm{ES}}_{\mathrm{active}}$ is negligible.


\paragraph{ES-active subspaces versus active subspaces:} Our mechanism for constructing $\mathcal{L}^{\mathrm{ES}}_{\mathrm{active}}$ is inspired by the recent theory of active subspaces \cite{constantine}, developed to determine the most important directions in the space of input parameters of high-dimensional blackbox functions such as computer simulations. 

The \textit{active subspace} of a differentiable function $F:\mathbb{R}^{d} \rightarrow \mathbb{R}$, square-integrable with respect to the given probabilistic density function $\rho:\mathbb{R}^{d} \rightarrow \mathbb{R}$, is given as a linear subspace $\mathcal{L}_{\mathrm{active}}$ defined by the first $r$ (for a fixed $r < d$)
eigenvectors of the following $d \times d$ symmetric positive definite matrix:
\begin{equation}
\mathrm{Cov} = \int_{\mathbf{x} \in \mathbb{R}^{d}} \nabla F(\mathbf{x}) \nabla F(\mathbf{x})^{\top} \rho(\mathbf{x}) d\mathbf{x}
\end{equation}
Density function $\rho$ determines where compact representation of $F$ is needed. In our approach we do not assume that $\nabla F$ exists, but the key difference between $\mathcal{L}^{\mathrm{ES}}_{\mathrm{active}}$ and $\mathcal{L}_{\mathrm{active}}$ lies somewhere else.

The goal of $\mathrm{ASEBO}$ is to avoid approximating the exact gradient $\nabla F (\mathbf{x}) \in \mathbb{R}^{d}$ which is what makes standard ES methods very expensive 
and which is done in \cite{constantine_3} via gradient sketching techniques combined with finite difference approaches (standard methods of choice for ES baselines). 
\vspace{-3mm}  
\begin{algorithm}[H]
\textbf{Hyperparameters:} smoothing parameter $\sigma$, horizon $C$, learning rate $\alpha$, probability regularizer $\beta$, initial probability parameter $q^{t}_{0} \in (0,1)$.\\
\textbf{Input: }subspaces: $\mathcal{L}^{\mathrm{ES}}_{\mathrm{active}}$,
$\mathcal{L}_{\mathrm{active}}^{\mathrm{ES},\perp}$, function $F$, vector $\theta_{t}$ \\
\textbf{Output:}  \\
\For{$l = 1, \cdots , C+1$}{
1. Compute $p^{t}_{l-1} =(1-2\beta) q^{t}_{l-1} + \beta$
and sample $a^{t}_l\sim \mathrm{Ber}(p^{t}_{l})$. \\
3. If $a^{t}_l = 1$, sample $\mathbf{g}_l \sim \mathcal{N}(0, \sigma \mathbf{I}_{\mathcal{L}_{\mathrm{active}}^{\mathrm{ES}}})$, 
   otherwise sample $\mathbf{g}_l \sim \mathcal{N}(0, \sigma \mathbf{I}_{\mathcal{L}_{\mathrm{active}}^{\mathrm{ES}, \perp}})$.\\
4. Compute $v_l = \frac{1}{2\sigma}\left(  F(\theta_t + \mathbf{g}_l) - F(\theta_t - \mathbf{g}_l)  \right) $.\\
5. Set $\mathbf{e}_l =(1-2\beta)\begin{bmatrix} \left(-\frac{a^{t}_l(\mathrm{dim}(\mathcal{L}_{\mathrm{active}}^{\mathrm{ES}})+2)}{(p^{t}_l)^{3}} \right) \\
\left(-\frac{(1-a^{t}_l)(\mathrm{dim}(\mathcal{L}^{\mathrm{ES}, \perp}_{\mathrm{active}}) +2)}{(1-p^{t}_l)^3} \right) \end{bmatrix} v_l^2 $.\\
6.  Set $q^{t}_{l} = \frac{q^{t}_{l-1} \exp(-\alpha \mathbf{e}_l(1))}{ q^{t}_{l-1} \exp(-\alpha \mathbf{e}_l(1))   + (1-q^{t}_{l-1}) \exp(-\alpha \mathbf{e}_l(2))} $.
}
\textbf{Return:} $p_{C}$.
 \caption{Explore estimator via exponentiated sampling}
\label{Alg: exponentiated_sampling}
\end{algorithm}
\vspace{-3mm}  
Instead, in $\mathrm{ASEBO}$ an ES-active subspace $\mathcal{L}_{\mathrm{active}}^{\mathrm{ES}}$ is itself used to define sensing directions and the number of chosen 
samples $k$ is given by the dimensionality of $\mathcal{L}_{\mathrm{active}}^{\mathrm{ES}}$. 
This drastically reduces sampling complexity, but comes at a price of risking the optimization to be trapped in the 
fixed lower-dimensional space that will not be representative for gradient data as optimization progresses. 
We propose a solution requiring only a constant number of extra queries to $F$ in the next sections.

\subsection{Exploration-exploitation trade-off: Adaptive Exploration Mechanism}  
\label{sec:exploration}
The procedure described above needs to be accompanied with an exploration strategy that will determine how frequently to 
choose sensing directions outside the constructed on-the-fly lower-dimensional ES-subspace $\mathcal{L}_{\mathrm{active}}^{\mathrm{ES}}$. 
Our exploration strategies will be encoded by hybrid probabilistic distributions for sampling sensing directions. 
The frequency of sensing from the distributions with covariance matrices obtained from $\mathcal{L}^{\mathrm{ES}}_{\mathrm{active}}$ (corresponding to exploitation) 
and from its orthogonal complement $\mathcal{L}^{\mathrm{ES}, \perp}_{\mathrm{active}}$ or entire space (corresponding to exploration) will be given by weights encoding the 
importance of exploitation versus exploration in any given iteration of the optimization. 
For a vector $\mathbf{x} \in \mathbb{R}^{d}$ denote by $\mathbf{x}_{\mathrm{active}}$
its projection onto $\mathcal{L}^{\mathrm{ES}}_{\mathrm{active}}$ and by $\mathbf{x}_{\perp}$ its projection onto $\mathcal{L}^{\mathrm{ES},\perp}_{\mathrm{active}}$.
The useful metric that can be used to update the above weights in an online manner in the $t^{th}$ iteration of the algorithm is the ratio: 
$
r = \frac{\|(\nabla F_{\sigma}(\theta_{t}))_{\mathrm{active}}\|_{2}}{\|(\nabla F_{\sigma}(\theta_{t}))_{\mathrm{\perp}}\|_{2}}.
$
Smaller values of $r$ indicate that current active subspace is not representative enough for the gradient and more aggressive exploration needs to be conducted.
In practice, we do not compute $r$ explicitly, but rather its approximated version $\widehat{r}$. 

One can simply take: 
$\widehat{r} = \frac{\|(\widehat\nabla^{\mathrm{AT}}_{\mathrm{MC}}F_{\sigma}(\theta_{t-1}))_{\mathrm{active}}\|_{2}}
{\|(\widehat\nabla^{\mathrm{AT}}_{\mathrm{MC}}F_{\sigma}(\theta_{t-1}))_{\mathrm{\perp}}\|_{2}}$, where 
$\widehat\nabla^{\mathrm{AT}}_{\mathrm{MC}}F_{\sigma}(\theta_{t-1})$ is obtained in the previous iteration.
But we can do better. It suffices to separately estimate $\|(\nabla F_{\sigma}(\theta_{t}))_{\mathrm{active}}\|_{2}$
and $\|(\nabla F_{\sigma}(\theta_{t}))_{\mathrm{\perp}}\|_{2}$. However we do not aim to estimate $(\nabla F_{\sigma}(\theta_{t}))_{\mathrm{active}}$
and $(\nabla F_{\sigma}(\theta_{t}))_{\mathrm{\perp}}$. That would be equivalent to computing exact estimate of $\nabla F_{\sigma}(\theta_{t})$, defeating 
the purpose of $\mathrm{ASEBO}$. Instead, we note that estimating the length of the unknown high-dimensional vector is much simpler than
estimating the vector itself and can be done in the probabilistic manner with arbitrary precision via the set of dot-product queries of size independent 
from dimensionality $d$ via compressed sensing methods. We refine this approach and propose more accurate contextual 
bandits method that also relies on dot-product queries applied in the ES-context, but aims to directly approximate optimal probabilities of sampling from 
$\mathcal{L}^{\mathrm{ES}}_{\mathrm{active}}$ rather than approximating gradients components' lengths (see $\mathrm{Algorithm}$ 2 box, the compressed sensing baseline 
is presented in the Appendix). The related computational overhead is measured in constant number of extra function queries, negligible in practice.

\vspace{-3mm} 
\subsection{The Algorithm}
\label{sec:algo}    

$\mathrm{ASEBO}$ is given in the $\mathrm{Algorithm}$ 1 box.
The algorithm we apply to score relative importance of sampling from the ES-active subspace $\mathcal{L}^{\mathrm{ES}}_{\mathrm{active}}$
versus from outside $\mathcal{L}^{\mathrm{ES}}_{\mathrm{active}}$ is in the $\mathrm{Algorithm}$ 2 box.

As we have already mentioned, it uses bandits method do determine optimal probability of sampling from
$\mathcal{L}^{\mathrm{ES}}_{\mathrm{active}}$. In the next section we show that by using these techniques we can substantially reduce the variance of ES blackbox gradient estimators
if ES-active subspaces approximate the gradient data well (which is the case for RL blackbox functions as presented in Fig. \ref{fig:pca_motivation}).
Horizon lengths $C$ in Algorithm 2 which determines the number of extra function queries should be in practice chosen as small constants.
In each iteration of Algorithm 1 the number of function queries is proportional to the dimensionality of the ES-active subspace
$\mathcal{L}^{\mathrm{ES}}_{\mathrm{active}}$ rather than the original space. 
\vspace{-3mm}

\section{Theoretical Results}\label{sec:theory}

We provide here theoretical guarantees for the ASEBO sampling mechanism (in Algorithm 1), where sensing directions $\{ \mathbf{g}_i\}$ at time $t$ are sampled from the hybrid distribution $\widehat{P}$:
with probability $p^{t}$ from $\mathcal{N}(0, \mathbf{I}_{\mathcal{L}_{\mathrm{active}}})$ and with probability $1-p^{t}$ from $\mathcal{N}(0, \mathbf{I}_{\mathcal{L}_{\mathrm{active}}^\perp} )$.


Following notation in Algorithm 1, let $\mathbf{U}^{\mathrm{act}} \in \mathbb{R}^{d\times r}$ be 
obtained by stacking together vectors of some orthonormal basis of $\mathcal{L}^{\mathrm{ES}}_{\mathrm{active}}$,
where $\mathrm{dim}(\mathcal{L}^{\mathrm{ES}}_{\mathrm{active}}) = r$ and let $\mathbf{U}^{\perp} \in \mathbb{R}^{d\times (d-r)}$
be obtained my stacking together vectors of some orthonormal basis of its orthogonal complement $\mathcal{L}^{\mathrm{ES}, \perp}_{\mathrm{active}}$.
Denote by $\sigma$ a smoothing parameter. We make the following regularity assumptions on $F$:
\begin{itemize}[noitemsep,topsep=0pt, leftmargin=*]
    \item[] \textbf{Assumption 1}. $F$ is $L-$Lipschitz, i.e. for all $\theta, \theta' \in \mathbb{R}^d$, $|F(\theta) - F(\theta')| \leq L \| \theta - \theta'\|_{2}$.
    \item[] \textbf{Assumption 2}. $F$ has a $\tau$-smooth third order derivative tensor with respect to $\sigma>0$, 
    so that $F(\theta + \sigma \mathbf{g}) = F(\theta) + \sigma \nabla F(\theta)^\top 
    \mathbf{g} + \frac{\sigma^2}{2}\mathbf{g}^\top H(\theta )\mathbf{g} + \frac{1}{6} \sigma^3 F'''(\theta)[\mathbf{v},\mathbf{v},\mathbf{v}]$ 
    for some $\mathbf{v} \in \mathbb{R}^{d}$ ($\|\mathbf{v}\|_{2} \leq \|\mathbf{g}\|_{2}$) 
    satisfying  $|F'''(\theta)[\mathbf{v},\mathbf{v},\mathbf{v}] \leq \tau \|\mathbf{v} \|_{2}^3 \leq \tau \| \mathbf{g}\|_{2}^3$.
\end{itemize}

Observe that:
$
\mathbb{E}_{\mathbf{g} \sim \widehat{P} } \left[   \mathbf{g}\mathbf{g}^\top         \right] = \left( p^t \mathbf{U}^{\mathrm{act}}\left(\mathbf{U}^{\mathrm{act}}   \right)^\top + (1-p^t) \mathbf{U}^{\perp}\left(  \mathbf{U}^{\perp} \right)^\top \right)$. Define $C_{1}=\left( p^t \mathbf{U}^{\mathrm{act}}\left(\mathbf{U}^{\mathrm{act}}   \right)^\top + (1-p^t) \mathbf{U}^{\perp}\left(  \mathbf{U}^{\perp} \right)^\top \right)$.
Let $\widehat{\nabla}^{\mathrm{AT},\mathrm{asebo}}_{\mathrm{MC},k=1}F_{\sigma}(\theta) = \mathbf{C}_1^{-1} \frac{F(\theta + \sigma \mathbf{g})\mathbf{g} + F(\theta + \sigma \mathbf{g}) (-\mathbf{g})}{2\sigma} $ 
be the gradient estimator corresponding to $\widehat{P}$. We will assume that $\sigma$ is small enough, i.e. $\sigma < \frac{1}{35}\sqrt{\frac{\epsilon \min(p^t, 1-p^t)}{\tau  d^3 \max(L, 1)} }$ for some precision parameter $\epsilon > 0$. Our first result shows that under these assumptions, baseline and $\mathrm{ASEBO}$ estimators of $\nabla_{\sigma}F(\theta)$ are also good estimators of $\nabla F(\theta)$:

\begin{lemma}\label{lemma::gradients_approximation}
If $F$ satisfies Assumptions 1 and 2, the estimators $\widehat{\nabla}^{\mathrm{AT},\mathrm{base}}_{\mathrm{MC},k=1}F_{\sigma}(\theta)$ and 
$\widehat{\nabla}^{\mathrm{AT},\mathrm{asebo}}_{\mathrm{MC},k=1}F_{\sigma}(\theta)$ are close to the true gradient $\nabla F(\theta)$, i.e.:
$\left \| \mathbb{E}_{\mathbf{g} \sim \mathcal{N}(0, \mathbf{I}_{d})}\left[  \widehat{\nabla}^{\mathrm{AT},\mathrm{base}}_{\mathrm{MC},k=1}F_{\sigma}(\theta)\right] - \nabla F(\theta) \right \| \leq \epsilon$ and $\left \| \mathbb{E}_{\mathbf{g} \sim \widehat{P}}\left[  \widehat{\nabla}^{\mathrm{AT},\mathrm{asebo}}_{\mathrm{MC},k=1}F_{\sigma}(\theta)\right] - \nabla F(\theta) \right \| \leq\epsilon$. 
\end{lemma}





               

\subsection{Variance reduction via non isotropic sampling}
We show now that under sampling strategy given by distribution $\widehat{P}$, the variance of the gradient estimator can be made smaller by choosing the probability parameter $p^t$ appropriately. 
Denote: $d_{\mathrm{active}} = \mathrm{dim}(\mathcal{L}_{\mathrm{active}}^{\mathrm{ES}})$ and $d_{\perp} = \mathrm{dim}(\mathcal{L}_{\mathrm{active}}^{\mathrm{ES}, \perp})$.
Let $\Gamma := \frac{d_{\mathrm{active}}+2}{p^t} s_{\mathbf{U}^{\mathrm{act}}} + \frac{d_{\perp} + 2}{1-p^t} s_{\mathbf{U}^{\perp}} - \|\nabla F(\theta)\|^2$.

\begin{theorem}\label{theorem:combined_theorem_variance1}
The following holds for $s_{\mathbf{U}^{\mathrm{act}}} = \| (\mathbf{U}^{\mathrm{act}})^\top \nabla F(\theta)  \|_{2}^2$ and $s_{\mathbf{U}^{\perp}}  = \|  (\mathbf{U}^{\perp})^\top \nabla F(\theta)  \|_{2}^2 $: 
\begin{enumerate}[noitemsep,topsep=0pt, leftmargin=*]
    \item The variance of ${\widehat{\nabla}^{\mathrm{AT},\mathrm{asebo}}_{\mathrm{MC},k=1}F_{\sigma}(\theta)}$ is close to $\Gamma$, i.e. 
$| \mathrm{Var}[{\widehat{\nabla}^{\mathrm{AT},\mathrm{asebo}}_{\mathrm{MC},k=1}F_{\sigma}(\theta)}] - \Gamma | \leq \epsilon$.
\item The choice of $p^t$ that minimizes $\Gamma$ satisfies $p^t_* := \frac{ \sqrt{ (s_{\mathbf{U}^{\mathrm{act}}})(d_{\mathrm{active}} + 2) } }{   \sqrt{(s_{\mathbf{U}^{\mathrm{act}}})(d_{\mathrm{active}} + 2)} +  \sqrt{(s_{\mathbf{U}^{\perp}})(d_{\mathbf{U}^{\perp}} + 2)}}$ and the optimal variance $\mathrm{Var}_{\mathrm{opt}}$ corresponding to $p^t_*$ satisfies:
$|\mathrm{Var}_{\mathrm{opt}} - \Delta| \leq \epsilon$ 
for $\Delta =      \left[  \sqrt{(s_{\mathbf{U}^{\mathrm{act}}})(d_{\mathrm{active}} + 2)} +    \sqrt{(s_{\mathbf{U}^{\perp}})(d_{\perp} + 2)} \right]^2 - \| \nabla F(\theta) \|^2$.
\item  $\mathrm{Var}_{\mathrm{opt}} \leq \mathrm{Var}[{\widehat{\nabla}^{\mathrm{AT},\mathrm{base}}_{\mathrm{MC},k=1}F_{\sigma}(\theta)}] + \epsilon -  \underbrace{| \sqrt{(s_{\mathbf{U}^{\perp}})(d_{\mathrm{active}} + 2)} -  \sqrt{(s_{\mathbf{U}^{\mathrm{act}}})(d_{\perp} + 2)}  |^2- 2 \| \nabla F(\theta)\|^2}_{\lambda}.$
Furthermore, slack variable $\lambda$ is always nonnegative.


\end{enumerate}
\end{theorem}

Theorem implies that when $s_{\mathbf{U}^{\mathrm{act}}} = (1-\alpha)\| \nabla F(\theta) \|_2^2$ and $s_{\mathbf{U}^\perp} = \alpha \| \nabla F(\theta) \|_2^2$, for some $\alpha \in (0,1)$, we have: $\mathrm{Var}[{{\widehat{\nabla}^{\mathrm{AT},\mathrm{base}}_{\mathrm{MC},k=1}F_{\sigma}(\theta)}}] \approx (d+1)\| \nabla F(\theta)\|^2$ whereas $\mathrm{Var}_{\mathrm{opt}} = \mathcal{O}\left( (1-\alpha) (d_{\mathrm{active}} + 1) + \alpha(d_{\perp} + 1)  \right)$. When $d_{\mathrm{active}} << d$ and $\alpha << 1$: $\mathrm{Var}_{\mathrm{opt}} \ll \mathrm{Var}[{{\widehat{\nabla}^{\mathrm{AT},\mathrm{base}}_{\mathrm{MC},k=1}F_{\sigma}(\theta)}}]$.

\subsection{Adaptive Mirror Descent}
In Theorem \ref{theorem:combined_theorem_variance1} we showed that for appropriate choices of $\mathcal{L}^{\mathrm{ES}}_{\mathrm{active}}$ and $p_{t}$, the 
gradient estimator ${\widehat{\nabla}^{\mathrm{AT},\mathrm{asebo}}_{\mathrm{MC},k=1}F_{\sigma}(\theta)}$ will have
significantly smaller variance than ${\widehat{\nabla}^{\mathrm{AT},\mathrm{base}}_{\mathrm{MC},k=1}F_{\sigma}(\theta)}$. 
In this section we show that Algorithm \ref{Alg: exponentiated_sampling} provides an adaptive way to choose $p^t$. 
Using tools from online learning theory, we provide regret guarantees that quantify the rate at which this 
Algorithm \ref{Alg: exponentiated_sampling} minimizes the variance of ${\widehat{\nabla}^{\mathrm{AT},\mathrm{asebo}}_{\mathrm{MC},k=1}F_{\sigma}(\theta)}$ 
and converges to the optimal $p^t_*$.

Let $\mathbf{p}^t_{l} = \binom{p^t_{l}}{1-p^t_{l}}$. The main component $\Gamma$ of 
the variance of ${\widehat{\nabla}^{\mathrm{AT},\mathrm{asebo}}_{\mathrm{MC},k=1}F_{\sigma}(\theta)}$ as a function of $\mathbf{p}^t_l$ equals  $\Gamma = \ell(\mathbf{p}^t_l) =\frac{d_{\mathrm{active}} + 2}{\mathbf{p}^t_l(1)} s_{\mathbf{U}^{\mathrm{act}}} + \frac{ d_{\perp} +2}{\mathbf{p}^t_l(2)} s_{\mathbf{U}^{\perp}} - \|\nabla F(\theta) \|^2$ (Theorem \ref{theorem:combined_theorem_variance1}). We have:

\begin{theorem}\label{theorem::regret_mirror_descent}
Let $\Delta_2$ be the a 2-d simplex. Under assumptions: 1 and 2, if
$\sigma < \frac{1}{35}\sqrt{\frac{\epsilon \min(p^t, 1-p^t)}{\tau  d^3 \max(L, 1)} }$, $\alpha = \frac{2\beta^2}{\sqrt{   C[(d_{\mathrm{active}}+2)^2 s_{\mathbf{U}^{\mathrm{act}}}^2+(d_{\mathrm{\perp}}+2)s_{\mathbf{U}^{\perp}}^2] }}$ and $\epsilon = \frac{\beta^3}{2C (d+1)}$, Algorithm \ref{Alg: exponentiated_sampling} satisfies:
\begin{equation*}
    \frac{1}{C}\mathbb{E}\left[  \sum_{l=1}^C   \ell(\mathbf{p}^t_l  )  \right]  -\min_{\mathbf{p} \in \beta + (1-2\beta)\Delta_2} \ell(\mathbf{p}) \leq \frac{\mathrm{Var}_{\mathrm{opt}}}{\beta^2 \sqrt{C}}+ \frac{1}{C}
\end{equation*}
\end{theorem}

\section{Experiments}\label{sec:experiments_overleaf}

In our experiments we use different classes of high-dimensional blackbox functions: RL blackbox functions (where the input is a high-dimensional vector encoding a neural network policy
$\pi: \mathcal{S} \rightarrow \mathcal{A}$ mapping states $s$ to actions $a$ and the output is the cumulative reward obtained by an agent
applying this policy in a particular environment) and functions from the recently open-sourced $\mathrm{Nevergrad}$ library \citep{nevergrad}. In practice one can setup the hyperparameters used by Algorithm 2 as follows: $\sigma = 0.01, C=10, \alpha=0.01, \beta=0.1, q_{0}^{t}=0.1$. 
For each algorithm we used $k=5$ seeds and obtained curves are median-curves with inter-quartile ranges presented as shadowed regions.

\vspace{-3mm} 
\begin{figure}[H]
\begin{minipage}{0.99\textwidth}
	\subfigure{\includegraphics[keepaspectratio, width=0.99\textwidth]{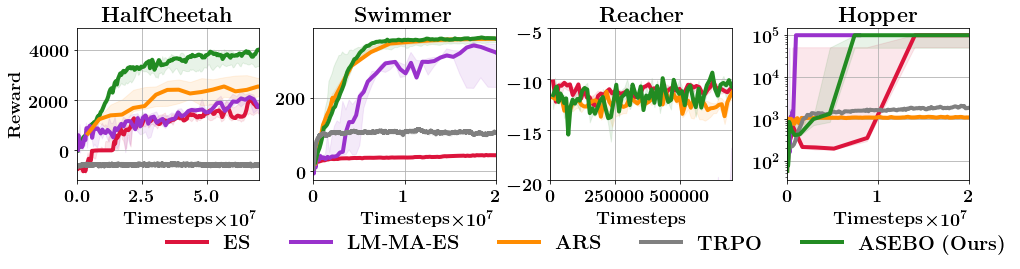}}
\end{minipage}
	\caption{Comparison of different blackbox optimization algorithms on $\mathrm{OpenAI}$ $\mathrm{Gym}$ tasks.
	 All curves are median-curves obtained from $k=5$ seeds and with inter-quartile ranges presented as shadowed regions. 
	 For Reacher we present only 3 curves since LM-MA-ES and TRPO did not learn.}
	\label{fig:openai_exps} 
\end{figure}
\vspace{-3mm} 
\subsection{RL blackbox functions}
We used the following environments from the $\mathrm{OpenAI}$ $\mathrm{Gym}$ library: 
$\mathrm{Swimmer}$-$\mathrm{v2}$, $\mathrm{HalfCheetah}$-$\mathrm{v2}$, $\mathrm{Walker2d}$-$\mathrm{v2}$, 
$\mathrm{Reacher}$-$\mathrm{v2}$, 
$\mathrm{Pusher}$-$\mathrm{v2}$ and
$\mathrm{Thrower}$-$\mathrm{v2}$. 
In all experiments we used policies encoded by neural network architectures of two hidden layers and with $\mathrm{tanh}$ nonlinearities, with $>100$ parameters. For gradient-based optimization we use $\mathrm{Adam}$.
For this class of blackbox functions we compared $\mathrm{ASEBO}$ with other generic blackbox methods as well as those specializing in optimizing
RL blackbox functions $F$, namely:
 \textbf{(1)} $\mathrm{CMA}$-$\mathrm{ES}$ variants; we compare against two recently introduced algorithms designed for high-dimensional settings (we use the implementation of $\mathrm{VkD}$-$\mathrm{CMA}$-$\mathrm{ES}$ in the 
$\mathrm{pycma}$ open-source implementation from $\mathrm{https://github.com/CMA}$-$\mathrm{ES/pycma}$), and that of $\mathrm{LM}$-$\mathrm{MA}$-$\mathrm{ES}$ from \cite{challenges_esrl}), \textbf{(2)} Augmented Random Search ($\mathrm{ARS}$) \citep{horia} (we use implementation released by the authors at $\mathrm{http://github.com/modestyachts/ARS}$), \textbf{(3)} Proximal Policy Optimization ($\mathrm{PPO}$) \citep{schulman2017proximal}  and Trust Region Policy Optimization ($\mathrm{TRPO}$) \citep{trpo} (we use $\mathrm{OpenAI}$ baseline implementation \citep{baselines}).
The results for four environments are on Fig. \ref{fig:openai_exps}.

\vspace{-3mm} 
\begin{table}[H]
  \caption{Median rewards obtained across $k=5$ seeds for seven RL environments. For each environment the top two performing algorithms are bolded, while the bottom two are shown in red.}
  \label{rl_results}
  \centering
  \begin{adjustbox}{width=\textwidth}

      \begin{tabular}{l*9{c}}
        \toprule
        \multicolumn{9}{c}{\textbf{Median reward after \# timesteps}}                   \\
        \cmidrule(r){3-9}
        \textbf{Environment}     & \textbf{Timesteps}     & $\mathrm{ASEBO}$ & $\mathrm{ES}$ & $\mathrm{ARS}$  & $\mathrm{VkD}$-$\mathrm{CMA}$ & $\mathrm{LM}$-$\mathrm{MA}$ & $\mathrm{TRPO}$ & $\mathrm{PPO}$ \\
        \midrule
        $\mathrm{HalfCheetah}$ & $5.10^7$ & $\mathbf{3821}$ & 1530 & $\mathbf{2420}$ & \textcolor{red}{-144} & 1632 & \textcolor{red}{-512} & 1514    \\
        
        $\mathrm{Swimmer}$ & $10^7$  & $\mathbf{358}$ & \textcolor{red}{36} & 348 &  $\mathbf{367}$ & 297 & 110 & \textcolor{red}{52}   \\
     
        $\mathrm{Walker2d}$ & $5.10^7$  & $\mathbf{9941}$  & \textcolor{red}{347} & 1112 & \textcolor{red}{1} & $\mathbf{18065}$ & 3011 & 2377   \\ 
         
        $\mathrm{Hopper}$ & $10^7$  & $\mathbf{99949}$  & \textcolor{red}{626} & 1091 & \textcolor{red}{42} & $\mathbf{100199}$ & 1663 & 1310   \\ 

        $\mathrm{Reacher}$ & $10^5$  & $\mathbf{-11}$ & $\mathbf{-10}$ & -12 &  \textcolor{red}{-1391} & -173 & -112 & \textcolor{red}{-196}   \\   
    
        $\mathrm{Pusher}$ & $10^5$  & $\mathbf{-46}$  & -48 & $\mathbf{-45}$ &  \textcolor{red}{-1001} & \textcolor{red}{-467} & -120 & -316   \\       
    
        $\mathrm{Thrower}$ & $10^5$  & $\mathbf{-89}$  & -96 & -90 &  \textcolor{red}{-796} & \textcolor{red}{-737} & $\mathbf{-85}$ & -175   \\  
        \bottomrule
      \end{tabular}
  \end{adjustbox}
\end{table}
\vspace{-4mm}

Sampling complexity is measured in the number of timesteps (environment transitions) used by the algorithms. 
$\mathrm{ASEBO}$ is the only algorithm that performs consistently across all seven environments (see: Table \ref{rl_results}),
outperforming $\mathrm{CMA}$-$\mathrm{ES}$ variants on all tasks aside from $\mathrm{VkD}$-$\mathrm{CMA}$-$\mathrm{ES}$ on $\mathrm{Swimmer}$
and $\mathrm{LM}$-$\mathrm{MA}$-$\mathrm{ES}$ on $\mathrm{Walker2d}$. For environments such as $\mathrm{Reacher}$, $\mathrm{Thrower}$ and $\mathrm{Pusher}$, these methods perform poorly, drastically underperforming even $\mathrm{Vanilla}$ $\mathrm{ES}$. On Fig. \ref{fig:lmmaes}, we demonstrate the common problem of state-of-the-art CMA-ES methods: if the number of samples $n$ is not carefully tuned, the algorithm does not learn. $\mathrm{ASEBO}$ does not have this problem since $n$ is learned on-the-fly.
\vspace{-2mm} 
\begin{figure}[H]
\begin{minipage}{0.99\textwidth}
	\subfigure{\includegraphics[keepaspectratio, width=0.99\textwidth]{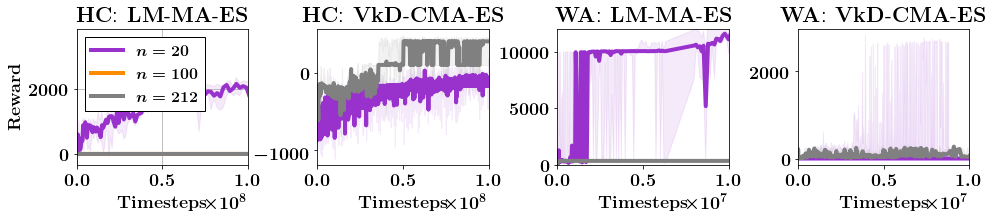}}	
\end{minipage}
	\caption{Sensitivity analysis for $\mathrm{CMA}$-$\mathrm{ES}$ variants on the $\mathrm{HalfCheetah}$ (HC) and $\mathrm{Walker2d}$ (WA) tasks. In each setting, we run $k=5$ seeds, solely changing the number of samples per iteration (or \textit{population size}) $n$.}
	\label{fig:lmmaes} 
\end{figure}
\vspace{-2mm} 

\vspace{-2mm} 
\subsection{$\mathrm{Nevergrad}$ blackbox functions}
We tested functions: $\mathrm{sphere}$, $\mathrm{rastrigin}$, $\mathrm{rosenbrock}$ and $\mathrm{lunacek}$ 
(from the class of Bi-Rastrigin/Lunacek's No.02 functions).
All tested functions are $1000$-dimensional.
The results are presented on Fig. \ref{fig:nevergrad_exps}.
$\mathrm{ASEBO}$ is the most reliable method across different functions.

\vspace{-3mm} 
\begin{figure}[H]
\begin{minipage}{0.99\textwidth}
	\subfigure{\includegraphics[keepaspectratio, width=0.99\textwidth]{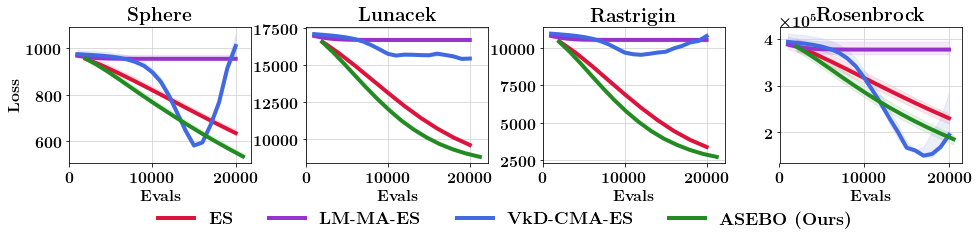}}
\end{minipage}
	\caption{Comparison of median-curves obtained from $k=5$ seeds for different algorithms on $\mathrm{Nevergrad}$ functions \citep{nevergrad}.
   Inter-quartile ranges are presented as shadowed regions.}
	\label{fig:nevergrad_exps} 
\end{figure}
\vspace{-3mm} 




\section{Conclusion}\label{sec:conclusions}

We proposed a new algorithm for optimizing high-dimensional blackbox functions.
$\mathrm{ASEBO}$ adjusts on-the-fly the strategy of choosing gradient sensing directions 
to the hardness of the problem at the current stage of optimization and can be applied for both RL and
non-RL problems. We provided theoretical guarantees for our method and exhaustive empirical validation.



\bibliographystyle{abbrv}
\bibliography{asebo}

\newpage
\onecolumn

\section*{APPENDIX: From Complexity to Simplicity: Adaptive ES-Active Subspaces for Blackbox Optimization}

\section{Theoretical Results}

Throughout this section we will assume the sensings directions $\{\mathbf{g}_i\}$ at time $t$ are sampled from one of the following families of distributions:
\begin{equation*}
    \widehat{P} = \begin{cases}
            \mathbf{g} \sim \mathcal{N}(0, \mathbf{I}_{\mathcal{L}^{\mathrm{ES}}_{\mathrm{active}}})& \text{ with probability } p^{t} \\
            \mathbf{g} \sim \mathcal{N}(0, \mathbf{I}_{\mathcal{L}^{\mathrm{ES, \perp}}_{\mathrm{active}}}) & \text{ with probability } 1-p^{t}
        \end{cases}
\end{equation*}

 Where $p^{t}$ is a probability parameter with values in $[0,1]$.

Denote an by $\mathbf{U}^{\mathrm{act}} \in \mathbb{R}^{d\times d_{\mathrm{active}}}$ an orthonormal basis of the active subspace $\mathcal{L}^{\mathrm{ES}}_{\mathrm{active}}$ and $\mathbf{U}^\perp \in \mathbb{R}^{d\times (d-d_{\mathrm{active}})}$ an orthonormal basis of $\mathcal{L}^{\mathrm{ES, \perp}}_{\mathrm{active}}$. 

Let's start by computing the covariance matrix of $\widehat{P}$:
\begin{equation*}
    \mathbb{E}_{\mathbf{g} \sim P_i } \left[   \mathbf{g}\mathbf{g}^\top         \right] = \underbrace{\left( p^{t} \mathbf{U}^{\mathrm{act}}(\mathbf{U}^{\mathrm{act}})^\top + (1-p^{t}) \mathbf{U}^{\perp}(\mathbf{U}^{\perp})^\top          \right)}_{\mathbf{C}_1}    
\end{equation*}

In order to simplify the notation of the proofs in this section we use the following conventions:

\begin{align*}
z_{ES} = \widehat{\nabla}^{\mathrm{AT},\mathrm{base}}_{\mathrm{MC},k=1}F_{\sigma}(\theta) \\
z_1 = \widehat{\nabla}^{\mathrm{AT},\mathrm{asebo}}_{\mathrm{MC},k=1}F_{\sigma}(\theta)
\end{align*}

Where $z_1$ is the ASEBO gradient estimator resulting form using sampling mechanism $\widehat{P}$. 

\paragraph{Notational simplification} To simplify notation we also use $\mathbf{U}$ instead of $\mathbf{U}^{\mathrm{act}}$, $\mathbf{I}_{\mathbf{U}}$ instead of $\mathbf{I}_{\mathcal{L}^{\mathrm{ES}}_{\mathrm{active}}}$ and $\mathbf{I}_{\mathbf{U}^\perp}$ instead of $\mathbf{I}_{\mathcal{L}^{\mathrm{ES}, \perp}_{\mathrm{active}}}$

Let $\epsilon >0$ be the precision parameter. We choose $\sigma$ with the goal of making the bias between the expectation of our gradient estimators and the true gradient of $F$ smaller than $\epsilon$. Throughout this section we assume $\sigma$ is small enough:

\begin{equation*}
    0< \sigma < \frac{1}{35}\sqrt{\frac{\epsilon \min(p^{t}, 1-p^{t})}{\tau  d^3 \max(L, 1)} }
\end{equation*}

\subsection{Gradient Estimators, their bias and their variance.}

In this section we aim to produce theoretical guarantees regarding the bias and variance of our proposed gradient estimators. We show that under the right assumptions, the isotropic and non isotropic versions of the evolution Strategies estimators have small bias, and 

We make the following assumptions on $F$:
\begin{itemize}
    \item[] \textbf{Assumption 1}. $F$ is $L-$Lipschitz. For all $\theta, \theta' \in \mathbb{R}^d$, $|F(\theta) - F(\theta')| \leq L \| \theta - \theta'\|$.
    \item[] \textbf{Assumption 2}. $F$ has a $\tau$-smooth third order derivative tensor, so that $F(\theta + \sigma \mathbf{g}) = F(\theta) + \sigma \nabla F(\theta)^\top \mathbf{g} + \frac{\sigma^2}{2}\mathbf{g}^\top H(\theta )\mathbf{g} + \frac{1}{6} \sigma^3 F'''(\theta)[v,v,v]$ with $v \in [0, \mathbf{g}]$ satisfying  $|F'''(\theta)[v,v,v] \leq \tau \|v \|^3 \leq \tau \| \mathbf{g}\|^3$.
\end{itemize}

Let $d_{\mathrm{active}}$ and $d_{\perp}$ denote the dimensionality of $\mathcal{L}_{(active)}$ and $\mathcal{L}_\perp$ respectively. 

Under these assumptions, $ \frac{F(\theta_t + \sigma \mathbf{g}) - F(\theta_t - \sigma \mathbf{g})}{2\sigma} = \left( \mathbf{g}^\top \nabla F(\theta_t ) \right)+  \xi_{\mathbf{g}}(\theta_t) $ such that $ \xi_\mathbf{g}(\theta_t) \leq \frac{\tau}{6} \sigma^2 \| \mathbf{g}\|^3$, uniformly over all $\theta_t$. We relax the constants slightly. If $F$'s third order derivative tensor is smooth with constant $\tau$:

\begin{equation*}
    \left|\frac{F(\theta_t + \sigma \mathbf{g}) - F(\theta_t - \sigma \mathbf{g})}{2\sigma} - \mathbf{g}^\top \nabla F(\theta_t) \right| \leq \tau\sigma^2 \|\mathbf{g}\|^3.
\end{equation*}

Recall the following definitions:

\begin{itemize}
    \item \textbf{Evolution Strategies Gradient}. Let $\mathbf{g} \sim \mathcal{N}(0, \mathbf{I})$. The ES gradient is defined as $z_{ES} = \frac{F(\theta_t + \sigma \mathbf{g}) -F(\theta_t - \sigma \mathbf{g}) }{2\sigma}\mathbf{g}$.
    \item \textbf{ $\widehat{P}$ Nonisotropic Gradient.}. Let $\mathbf{g} \sim \widehat{P}$. The $\widehat{P}$ gradient is defined as $z_1 = \mathbf{C}_1^{-1}\frac{F(\theta_t + \sigma \mathbf{g}) -F(\theta_t - \sigma \mathbf{g}) }{2\sigma}\mathbf{g}$.
\end{itemize}

The following inequalitites hold:

\begin{align*}
    \| \xi_\mathbf{g}(\theta_t) \mathbf{g} \|^2 &\leq \frac{\tau}{6} \sigma^2 \| \mathbf{g} \|^4 \\
    \|\mathbb{E}_{\mathbf{g} \sim \mathcal{N}(0, \mathbf{I})} \left[   \xi_\mathbf{g}(\theta_t) \mathbf{g}  \right] \|^2 &\leq \frac{\sigma^4 \tau^2 }{36} \left(   \mathbb{E}_{\mathbf{g} \sim \mathcal{N}(0, \mathbf{I})}\left[ \|\mathbf{g}\|^4   \right] \right)^2\leq \frac{\sigma^4 \tau^2 d^4 }{4} \\
     \|\mathbb{E}_{\mathbf{g} \sim \mathcal{N}(0, \mathbf{I}_{\mathbf{U}})} \left[   \xi_\mathbf{g}(\theta_t) \mathbf{g}  \right] \|^2 &\leq \frac{\sigma^4 \tau^2 }{36} \left(\mathbb{E}_{\mathbf{g} \sim \mathcal{N}(0, \mathbf{I}_{\mathbf{U}^\perp})} \left[ \|\mathbf{g}\|^4   \right] \right)^2 \leq \frac{\sigma^4 \tau^2 d_{\mathrm{active}}^4 }{4} \\
          \|\mathbb{E}_{\mathbf{g} \sim \mathcal{N}(0, \mathbf{I}_\mathbf{U})} \left[   \xi_\mathbf{g}(\theta_t) \mathbf{g}  \right] \|^2 &\leq \frac{\sigma^4 \tau^2 }{36} \left(\mathbb{E}_{\mathbf{g} \sim \mathcal{N}(0, \mathbf{I}_{\mathbf{U}^\perp}) } \left[ \|\mathbf{g}\|^4   \right] \right)^2\leq \frac{\sigma^4 \tau^2 d_{\perp}^4 }{4}
\end{align*}


\textbf{Bounding the Bias}
The first result in this section is to show that under the right conditions the ES gradient estimators in both the isotropic and non isotropic cases can be close to the true gradient provided the function satisfies Assumptions 1 and 2. Theorem \ref{theorem::es_gradient_bias} deals with the isotropic case and Theorem \ref{theorem::es_gradient_bias_non_isotropic} with the non isotropic case. The combination of these results yields the proof of Lemma \ref{lemma::gradients_approximation} in the main text.

\begin{theorem}\label{theorem::es_gradient_bias}
The evolution strategies gradient estimator $z_{ES}$ satisfies:

\begin{equation}\label{eq::es_gradient_bias}
    \left \| \mathbb{E}_{\mathbf{g} \sim \mathcal{N}(0, \mathbf{I})}\left[ z_{ES}\right] - \nabla F(\theta_t) \right \| \leq3 \tau\sigma^2d^2
\end{equation}

If $\sigma < \frac{1}{35}\sqrt{\frac{\epsilon \min(p^{t}, 1-p^{t})}{\tau  d^3 \max(L, 1)} }$:
\begin{equation}
    \left \| \mathbb{E}_{\mathbf{g} \sim \mathcal{N}(0, \mathbf{I})}\left[ z_{ES}\right] - \nabla F(\theta_t) \right \| \leq \epsilon
\end{equation}

\end{theorem}

\begin{proof}
Notice that $\| \mathbf{g} \|^4 = (\sum_{i=1}^d \mathbf{g}(i)^2 )^2 \leq d\sum_{i=1}^d \mathbf{g}(i)^4 $. Where we denote $\mathbf{g}(i)$ as the $i-$th entry of the $d-$dimensional vector $\mathbf{g}\in \mathbb{R}^d$. Since $\mathbb{E}[\mathbf{g}(i)^4] = 3$ for all $i$:
\begin{equation*}
    \mathbb{E}_{\mathbf{g} \sim \mathcal{N}(0, \mathbf{I})} \left[   \| \mathbf{g}\|^4 \right] \leq 3 d^2
\end{equation*}
And therefore:
\begin{equation*}
       \left \| \mathbb{E}_{\mathbf{g} \sim \mathcal{N}(0, \mathbf{I})}\left[  \frac{F(\theta_t + \sigma \mathbf{g}) - F(\theta_t - \sigma \mathbf{g}) }{ 2\sigma}    \mathbf{g}\right] - \nabla F(\theta_t)  \right\| \leq \tau\sigma^2 \mathbb{E}_{\mathbf{g} \sim \mathcal{N}(0, \mathbf{I})} \left[    \| \mathbf{g}\|^4 \right] \leq 3\tau\sigma^2 d^2
\end{equation*}
\end{proof}

A similar result holds for the $z_1$ gradient.

\begin{theorem}\label{theorem::es_gradient_bias_non_isotropic}
The non isotropic $\widehat{P}$ gradient estimator satisfies:
\begin{equation*}
    \| \mathbb{E}_{\mathbf{g} \sim \widehat{P}} \left[    z_1\right] - \nabla F(\theta_t) \| \leq \frac{3\sigma^2 \tau }{p^{t}}  d_{\mathrm{active}}^2 + \frac{3\sigma^2 \tau}{1-p^{t}}d_{\perp}^2
\end{equation*}

If $\sigma < \frac{1}{35}\sqrt{\frac{\epsilon \min(p^{t}, 1-p^{t})}{\tau  d^3 \max(L, 1)} }$:

\begin{equation*}
    \| \mathbb{E}_{\mathbf{g} \sim \widehat{P}} \left[    z_1\right] - \nabla F(\theta_t) \| \leq \epsilon
\end{equation*}

\end{theorem}

\begin{proof}
Expanding $\mathbb{E}_{\mathbf{g} \sim \widehat{P}} \left[  z_1 \right]$ yields:
\begin{align*}
    \mathbb{E}_{\mathbf{g} \sim \widehat{P}} \left[  z_1 \right] &= \mathbf{C}_1^{-1}\mathbb{E}_{\mathbf{g} \sim \widehat{P}} \left[  \frac{F(\theta_t + \sigma \mathbf{g}) - F(\theta_t - \sigma \mathbf{g})}{2\sigma} \mathbf{g}     \right] \\
    &= \mathbf{C}_1^{-1}\mathbb{E}_{\mathbf{g} \sim \widehat{P}} \left[    \mathbf{g}\mathbf{g}^\top \nabla F(\theta_t) + \xi_\mathbf{g}(\theta_t)\mathbf{g}        \right]\\
    &= \nabla F(\theta_t) + \frac{1}{p^{t}} \mathbb{E}_{\mathbf{g} \sim \mathcal{N}(0, \mathbf{I}_\mathbf{U})} \left[  \xi_\mathbf{g}(\theta_t) \mathbf{g}   \right] + \frac{1}{1-p^{t}} \mathbb{E}_{\mathbf{g} \sim \mathcal{N}(0, \mathbf{I}_{\mathbf{U}^\perp})} \left[  \xi_\mathbf{g}(\theta_t) \mathbf{g}   \right]
\end{align*}

By a similar argument as in the proof of Theorem \ref{theorem::es_gradient_bias}:

\begin{align*}
\| \mathbb{E}_{\mathbf{g} \sim \mathcal{N}(0, \mathbf{I}_\mathbf{U})} \left[  \xi_\mathbf{g}(\theta_t) \mathbf{g}   \right] \| \leq 3\tau \sigma^2 d_{\mathrm{active}}^2 \\
\| \mathbb{E}_{\mathbf{g} \sim \mathcal{N}(0, \mathbf{I}_{\mathbf{U}^\perp})} \left[  \xi_\mathbf{g}(\theta_t) \mathbf{g}   \right] \| \leq 3\tau \sigma^2 d_{\perp}^2
\end{align*}
The result follows.
\end{proof}

\paragraph{Towards bounding the variance} We start by showing how under the right assumptions the expected squared norm of the ES gradients are also bounded away from the squared norms of the true gradients. The distance between the square norms of the expectation of the ES gradient and the true gradient of $F$ are also bounded. Theorem \ref{theorem::expected_squared_norm_bound_ES} deals with the isotropic ES estimator and Theorem \ref{theorem::bias_bound_z1} with its non isotropic counterpart:

\begin{theorem}\label{theorem::expected_squared_norm_bound_ES}
If $F$ satisfies Assumption 1 and 2:
\begin{equation}\label{eq::bound_squared_expectation_squared_gradient}
    \left| \left\| \mathbb{E}_{\mathbf{g} \sim \mathcal{N}(0, \mathbf{I})} \left[     z_{ES} \right]\right \|^2 - \| \nabla F(\theta_t) \|^2\right| \leq 105\tau^2\sigma^4 d^4 + 6\tau\sigma^2 L  d^2
\end{equation}
If $\sigma < \frac{1}{35}\sqrt{\frac{\epsilon \min(p^{t}, 1-p^{t})}{\tau  d^3 \max(L, 1)} }$:

\begin{equation}
    \left| \left\| \mathbb{E}_{\mathbf{g} \sim \mathcal{N}(0, \mathbf{I})} \left[     z_{ES} \right]\right \|^2 - \| \nabla F(\theta_t) \|^2\right| \leq \epsilon
\end{equation}

\end{theorem}

\begin{proof}
\begin{align*}
\left| \left\| \mathbb{E}_{\mathbf{g} \sim \mathcal{N}(0, \mathbf{I})} \left[     \frac{F(\theta_t + \sigma \mathbf{g}) - F(\theta_t - \sigma \mathbf{g})}{2\sigma} \mathbf{g}  \right]\right \|^2 - \| \nabla F(\theta_t) \|^2\right|  &\leq \tau^2 \left( \sigma^2 \mathbb{E}_{\mathbf{g} \sim \mathcal{N}(0, \mathbf{I})} \left[ \|\mathbf{g} \|^4   \right]\right)^2 + \\
&2 \tau \sigma^2 \| \nabla F(\theta_t) \| \| \mathbb{E}_{\mathbf{g} \sim  \mathcal{N}(0, \mathbf{I})} \left[    \|\mathbf{g} \|^4 \right] \\
&\leq 105\tau^2\sigma^4 d^4 + 6\tau\sigma^2 L  d^2
\end{align*}
\end{proof}

\begin{theorem}\label{theorem::bias_bound_z1}
If $F$ satisfies Assumption 1 and 2:
\begin{align*}
    \left| \left\| \mathbb{E}_{\mathbf{g} \sim \widehat{P}} \left[     z_{1} \right]\right \|^2 - \| \nabla F(\theta_t) \|^2\right| &\leq \frac{1}{(p^{t})^2} \frac{\sigma^4 \tau^2 d_{\mathrm{active}}^4}{4} + \frac{1}{(1-p^{t})^2} \frac{\sigma^4 \tau^2 d_{\perp}^4}{4} + \frac{2}{p^{t}} L \frac{\sigma^2 \tau d_{\mathrm{active}}^2}{4} + \\
    &\frac{2}{1-p^{t}}L \frac{\sigma^2\tau d_{\perp}^2}{4} + \frac{2}{p^{t}(1-p^{t})} \frac{\sigma^4 \tau^2 d_{\mathrm{active}}^2 d_{\perp}^2}{16}
\end{align*}

If $\sigma < \frac{1}{35}\sqrt{\frac{\epsilon \min(p^{t}, 1-p^{t})}{\tau  d^3 \max(L, 1)} }$:

\begin{equation*}
    \left| \left\| \mathbb{E}_{\mathbf{g} \sim \widehat{P}} \left[     z_{1} \right]\right \|^2 - \| \nabla F(\theta_t) \|^2\right| \leq \epsilon
\end{equation*}

\end{theorem}

\begin{proof}

Consider the following expansion of $\mathbb{E}\left[    z_1 \right] $.

\begin{align*}
    \| \mathbb{E}_{\mathbf{g} \sim   \widehat{P} } \left[   z_1  \right] \|^2 &= \| \nabla F(\theta_t) \|^2  + \frac{1}{(p^{t})^2} \| \mathbb{E}_{\mathbf{g} \sim \mathcal{N}(0, \mathbf{I}_\mathbf{U})}\left[ \xi_\mathbf{g}(\theta_t)\mathbf{g}   \right] \|^2 + \left( \frac{1}{1-p^{t}}\right)^2 \| \mathbb{E}_{\mathbf{g} \sim \mathcal{N}(0, \mathbf{I}_{\mathbf{U}^\perp}) } \left[  \xi_\mathbf{g}(\theta_t)\mathbf{g}    \right] + \\
    &\frac{2}{p^{t}} \langle \nabla F(\theta_t) , \mathbb{E}_{\mathbf{g} \sim \mathcal{N}(0, \mathbf{I}_\mathbf{U})}\left[ \xi_\mathbf{g}(\theta_t)\mathbf{g}   \right] \rangle + \frac{2}{1-p^{t}} \langle \nabla F(\theta_t) , \mathbb{E}_{\mathbf{g} \sim \mathcal{N}(0, \mathbf{I}_{\mathbf{U}^\perp})}\left[ \xi_\mathbf{g}(\theta_t)\mathbf{g}   \right] \rangle + \\
    &\frac{2}{p^{t}(1-p^{t})} \langle \nabla      \mathbb{E}_{\mathbf{g} \sim \mathcal{N}(0, \mathbf{I}_\mathbf{U})}\left[ \xi_\mathbf{g}(\theta_t)\mathbf{g}   \right], \mathbb{E}_{\mathbf{g} \sim \mathcal{N}(0, \mathbf{I}_{\mathbf{U}^\perp})}\left[ \xi_\mathbf{g}(\theta_t)\mathbf{g}   \right] \rangle 
\end{align*}

And therefore by Cauchy Schwartz:

\begin{align*}
\left|  \| \mathbb{E}_{\mathbf{g} \sim   \widehat{P} } \left[   z_1  \right] \|^2    - \| \nabla  F(\theta_t) \|^2 \right| &\leq \frac{1}{(p^{t})^2} \frac{\sigma^4 \tau^2 d_{\mathrm{active}}^4}{4} + \frac{1}{(1-p^{t})^2} \frac{\sigma^4 \tau^2 d_{\perp}^4}{4} + \\
\frac{2}{p^{t}} L \frac{\sigma^2 \tau d_{\mathrm{active}}^2}{4} + \frac{2}{1-p^{t}}L \frac{\sigma^2\tau d_{\perp}^2}{4} + \frac{2}{p^{t}(1-p^{t})} \frac{\sigma^4 \tau^2 d_{\mathrm{active}}^2 d_{\perp}^2}{16}
\end{align*}

As desired.
\end{proof}
\textbf{Bounding the variance of $z_{ES}$ and $z_1$.} We have now the necessary ingredients for bounding the variance of the ES isotropic and non isotropic estimators. We start by showing in theorem \ref{theorem::variance_isotropic_ES} that the variance of the isotropic estimator is roughly of the order of $(d+1) \| \nabla F(\theta_t) \|^2$. In contrast, Theorem \ref{theorem::variance_non_isotropic_estimator} characterizes the variance of $z_1$ the non isotropic ES gradient estimator in terms of the $\nabla F(\theta_t)$ decomposition along the subspaces spanned by $\mathbf{U}$ and $\mathbf{U}^\perp$. In the following section \ref{section::variance_reduction_non_isotropic} we show that with an appropriate choice of the probabilities $p^t$, $1-p^t$, and provided the subspace decomposition is adequate, the variance of the non isotropic gradient estimator can be much smaller than the variance of the $z_{ES}$.

\begin{theorem}\label{theorem::variance_isotropic_ES}
If $F$ satisfies Assumption 1 and 2, the variance of the ES estimator satisfies:
\begin{equation*}
    |\text{Var}_{\text{ES} } - (d+1) \| \nabla F(\theta_t) \|^2 | \leq   105\tau^2\sigma^4 d^4 + 6\tau\sigma^2 L  d^2 + 15d^3\sigma^2 L \tau + 105 \tau^2 \sigma^4 d^4
\end{equation*}

If $\sigma < \frac{1}{35}\sqrt{\frac{\epsilon \min(p^{t}, 1-p^{t})}{\tau  d^3 \max(L, 1)} }$:

\begin{equation*}
    |\text{Var}_{\text{ES} } - (d+1) \| \nabla F(\theta_t) \|^2 | \leq \epsilon
\end{equation*}

\end{theorem}

\begin{proof}
The second moment of the ES estimator satisfies:
\begin{align*}
\mathbb{E}_{\mathbf{g} \sim \mathcal{N}(0, \mathbf{I})} \left[ z_{ES}^\top z_{ES}    \right] &= \mathbb{E}_{\mathbf{g} \sim \mathcal{N}(0, \mathbf{I})} \left[ \frac{ (F(\theta_t + \sigma \mathbf{g}) - F(\theta_t- \sigma \mathbf{g}))^2      )}{ 2^2\sigma^2}      \mathbf{g}^\top\mathbf{g} \right] \\
  &= \mathbb{E}_{\mathbf{g} \sim \mathcal{N}(0, \mathbf{I})} \left[ \left( \mathbf{g}^\top \nabla F(\theta_t) +\xi_\mathbf{g}(\theta_t)       \right)^2 \mathbf{g}^\top\mathbf{g}      \right] \\
&= \mathbb{E}_{\mathbf{g} \sim \mathcal{N}(0, \mathbf{I})} \left[ \nabla F(x_t)^\top \mathbf{g} \mathbf{g}^\top \mathbf{g} \mathbf{g}^\top \nabla F(\theta_t) + 2\nabla F(\theta_t)^\top \mathbf{g} \mathbf{g}^\top \mathbf{g} \xi_\mathbf{g}(\theta_t) + \xi_\mathbf{g}(\theta_t)^2 \mathbf{g}^\top \mathbf{g}       \right]\\
&= (d+2) \| \nabla F(\theta_t) \|^2 +  2\mathbb{E}_{\mathbf{g}\sim \mathcal{N}(0, \mathbf{I})} \left[      \nabla F(\theta_t)^\top \mathbf{g} \mathbf{g}^\top \mathbf{g} \xi_\mathbf{g}(\theta_t)   \right] +\\
&\mathbb{E}_{\mathbf{g} \sim \mathcal{N}(0, \mathbf{I})}\left[     \xi_\mathbf{g}(\theta_t)^2 \mathbf{g}^\top \mathbf{g} \right]
\end{align*}

Under Assumption 1 and 2, the following bound for the second and third terms of the last equality holds:

\begin{equation*}
   \left| \mathbb{E}_{\mathbf{g} \sim \mathcal{N}(0, \mathbf{I})}\left[   \nabla F(\theta_t)^\top \mathbf{g}\mathbf{g}^\top \mathbf{g} \xi_{\mathbf{g}}(\theta_t)  \right] \right| \leq \mathbb{E}_{\mathbf{g}\sim \mathcal{N}(0, \mathbf{I})}\left[  \| \nabla F(\theta_t) \| \|\mathbf{g}\|^6       \right]\leq 15d^3 \sigma^2 L \tau
\end{equation*}

And:

\begin{equation*}
  \left |  \mathbb{E}_{\mathbf{g} \sim \mathcal{N}(0, \mathbf{I})}\left[ \xi_{\mathbf{g} }(\theta_t)^2 \mathbf{g}^\top \mathbf{g}      \right] \right| \leq \tau^2 \sigma^4 \mathbb{E}_{\mathbf{g} \sim \mathcal{N}(0, \mathbf{I})} \left[\| \mathbf{g}\|^8     \right] \leq 105 \tau^2 \sigma^4 d^4
\end{equation*}

Therefore:

\begin{equation*}
    \text{Var}_{\text{ES}} = \underbrace{ \mathbb{E}_{\mathbf{g} \sim \mathcal{N}(0, \mathbf{I})} \left[ \frac{ (F(\theta_t + \sigma \mathbf{g}) - F(\theta_t- \sigma \mathbf{g}))^2      )}{ 2^2\sigma^2}      \mathbf{g}^\top\mathbf{g} \right]}_{\diamondsuit} - \underbrace{\left\|\mathbb{E}_{\mathbf{g} \sim \mathcal{N}(0, \mathbf{I})}\left[  \frac{F(\theta_t + \sigma \mathbf{g}) - F(\theta_t - \sigma \mathbf{g}) }{ 2\sigma}    \mathbf{g}\right] \right \|^2}_{\spadesuit}
\end{equation*}

After coalescing the bounds dervied in the preceeding section, we can obtain the following bound on the term $\diamondsuit$:

\begin{equation*}
\left| \diamondsuit - (d+2) \| \nabla F(\theta_t) \|^2 \right| \leq  15d^3\sigma^2 L \tau + 105 \tau^2 \sigma^4 d^4
\end{equation*}

Notice that by virtue of \ref{eq::bound_squared_expectation_squared_gradient}, the following bound on term $\spadesuit$ of the previous equation holds:

\begin{equation*}
  \left|\spadesuit  - \| \nabla F(\theta_t) \|^2 \right|  \leq     105\tau^2\sigma^4 d^4 + 6\tau\sigma^2 L  d^2
\end{equation*}

Combining these two inequalities the result follows.

\end{proof}

A similar theorem holds for $z_1$.

\begin{theorem}\label{theorem::variance_non_isotropic_estimator}

Let $\Gamma = \left( \frac{d_{\mathrm{active}}+2}{p^{t}} \| \mathbf{U}^\top \nabla F(\theta_t) \|^2 + \frac{d_{\perp} +2}{1-p^{t}} \| (\mathbf{U}^\perp)^\top F(\theta_t) \|^2 - \| \nabla F(\theta_t)\|^2 \right)$.

\begin{align*}
\left| Var_{\widehat{P}} - \Gamma \right | &\leq  \frac{1}{p^{t}}\left(   15d_{\mathrm{active}}^3\sigma^2 L \tau + 105 \tau^2 \sigma^4 d_{\mathrm{active}}^4  \right) +\\
&\frac{1}{1-p^{t}}\left(  15d_{\perp}^3\sigma^2 L \tau + 105 \tau^2 \sigma^4 d_{\perp}^4  \right) + \\
&  \frac{1}{(p^{t})^2} \frac{\sigma^4 \tau^2 d_{\mathrm{active}}^4}{4} + \frac{1}{(1-p^{t})^2} \frac{\sigma^4 \tau^2 d_{\perp}^4}{4} +  \\
&\frac{2}{p^{t}} L \frac{\sigma^2 \tau d_{\mathrm{active}}^2}{4} +\frac{2}{1-p^{t}}L \frac{\sigma^2\tau d_{\perp}^2}{4} + \\
&\frac{2}{p^{t}(1-p^{t})} \frac{\sigma^4 \tau^2 d_{\mathrm{active}}^2 d_{\perp}^2}{16}
\end{align*}

If $\sigma < \frac{1}{35}\sqrt{\frac{\epsilon \min(p^{t}, 1-p^{t})}{\tau  d^3 \max(L, 1)} }$:

\begin{equation*}
\left| Var_{\widehat{P}} - \left( \frac{d_{\mathrm{active}}+2}{p^{t}} \| \mathbf{U}^\top \nabla F(\theta_t) \|^2 + \frac{d_{\perp} +2}{1-p^{t}} \| (\mathbf{U}^\perp)^\top F(\theta_t) \|^2 - \| \nabla F(\theta_t)\|^2 \right) \right | \leq  \epsilon
\end{equation*}

\end{theorem}

\begin{proof}
The second moment of $z_1$ satisfies:

\begin{align*}
    \mathbb{E}_{\mathbf{g} \sim \widehat{P}} \left[ z_1^\top z_1   \right] &= \frac{1}{p^{t}} \mathbb{E}_{\mathbf{g} \sim \mathcal{N}(0, \mathbf{I}_\mathbf{U})} \left[  \frac{ (F(\theta_t + \sigma \mathbf{g}) - F(\theta_t- \sigma \mathbf{g}))^2      )}{ 2^2\sigma^2}      \mathbf{g}^\top\mathbf{g}    \right] + \\
    &\frac{1}{1-p^{t}}\mathbb{E}_{\mathbf{g} \sim \mathcal{N}(0, \mathbf{I}_{\mathbf{U}^\perp})}\left[    \frac{ (F(\theta_t + \sigma \mathbf{g}) - F(\theta_t- \sigma \mathbf{g}))^2      )}{ 2^2\sigma^2}      \mathbf{g}^\top\mathbf{g} \right]
\end{align*}

Notice that:

\begin{equation*}
    \text{Var}_{\widehat{P}} = \underbrace{ \mathbb{E}_{\mathbf{g} \sim \widehat{P}} \left[ \frac{ (F(\theta_t + \sigma \mathbf{g}) - F(\theta_t- \sigma \mathbf{g}))^2      )}{ 2^2\sigma^2}      \mathbf{g}^\top \mathbf{C}_1^{-2}\mathbf{g} \right]}_{\diamondsuit} - \underbrace{\left\|\mathbb{E}_{\mathbf{g} \sim \widehat{P}}\left[  \frac{F(\theta_t + \sigma \mathbf{g}) - F(\theta_t - \sigma \mathbf{g}) }{ 2\sigma}    C_1^{-1} \mathbf{g}\right] \right \|^2}_{\spadesuit}
\end{equation*}

By a similar argument as that in the previous theorem, we conclude:

\begin{align*}
    \left|\diamondsuit - \frac{d_{\mathrm{active}} + 2 }{ p^{t}} \| \mathbf{U}^\top \nabla F(\theta_t ) \|^2 - \frac{(d_{V_\perp } + 2)}{1-p^{t}} \| (\mathbf{U}^\perp)^\top \nabla F(\theta_t) \|^2 \right| &\leq \frac{1}{p^{t}}\left(   15d_{\mathrm{active}}^3\sigma^2 L \tau + 105 \tau^2 \sigma^4 d_{\mathrm{active}}^4  \right) +\\
    &\frac{1}{1-p^{t}}\left(  15d_{\perp}^3\sigma^2 L \tau + 105 \tau^2 \sigma^4 d_{\perp}^4  \right)
\end{align*}

By Theorem \ref{theorem::bias_bound_z1}:

\begin{align*}
    \left |\spadesuit -     \| \nabla F(\theta_t)  \|^2      \right| &\leq   \frac{1}{(p^{t})^2} \frac{\sigma^4 \tau^2 d_{\mathrm{active}}^4}{4} + \frac{1}{(1-p^{t})^2} \frac{\sigma^4 \tau^2 d_{\perp}^4}{4} + \frac{2}{p^{t}} L \frac{\sigma^2 \tau d_{\mathrm{active}}^2}{4} + \frac{2}{1-p^{t}}L \frac{\sigma^2\tau d_{\perp}^2}{4} + \\
    &\frac{2}{p^{t}(1-p^{t})} \frac{\sigma^4 \tau^2 d_{\mathrm{active}}^2 d_{\perp}^2}{16}
\end{align*}

\end{proof}

\subsection{Variance reduction via non isotropic sampling}\label{section::variance_reduction_non_isotropic}
The first result of this section is to condense the theorems in the previous sections into a single result (see Theorem \ref{theorem::condensed_results}). Lemma \ref{lemma::optimal_variance} then shows what the variance corresponding to the optimal choice of parameter $p^t$ is. Theorem \ref{theorem::condition_for_domination} then provides conditions under which the approximate variance (without considering the bias terms) corresponding to the optimal non isotropic estimator is smaller than the variance of the isotropic one. Finally Thoerem \ref{theorem::condensed_variance_results} takes into account the bias and states the final reuslt of this section. The combination of these results yield the proof of Theorem \ref{theorem:combined_theorem_variance1} in the main section of the paper.

\begin{theorem}\label{theorem::condensed_results}
Let $\epsilon > 0$. If $\sigma < \frac{1}{35}\sqrt{\frac{\epsilon \min(p^{t}, 1-p^{t})}{\tau  d^3 \max(L, 1)} }$ then:
\begin{align}
    \left \| \mathbb{E}_{\mathbf{g} \sim \mathcal{N}(0, \mathbf{I})}\left[ z_{ES}\right] - \nabla F(\theta_t) \right \| \leq \epsilon \\
        \left \| \mathbb{E}_{\mathbf{g} \sim \widehat{P}}\left[ z_{1}\right] - \nabla F(\theta_t) \right \| \leq \epsilon
\end{align}

and

\begin{align}
    \left|  \text{Var}_{ES} -  (d +1 )  \| \nabla F(\theta_t) \|^2     \right| \leq \epsilon \\
    \left|   \text{Var}_{\widehat{P}} - \left( \frac{d_{\mathrm{active}}+2}{p^{t}} \| \mathbf{U}^\top \nabla F(\theta_t) \|^2 + \frac{d_{\perp} +2}{1-p^{t}} \| (\mathbf{U}^\perp)^\top F(\theta_t) \|^2 - \| \nabla F(\theta_t)\|^2 \right)       \right| \leq \epsilon 
    \end{align}

\end{theorem}

We say that in this case:

\begin{equation*}
    \text{Var}_{\widehat{P}} \approx \underbrace{\left( \frac{d_{\mathrm{active}}+2}{p^{t}} \| \mathbf{U}^\top \nabla F(\theta_t) \|^2 + \frac{d_{\perp} +2}{1-p^{t}} \| (\mathbf{U}^\perp)^\top F(\theta_t) \|^2 - \| \nabla F(\theta_t)\|^2 \right) }_{\text{Var}^M_{\widehat{P}}}
\end{equation*}

and $\text{Var}_{ES} \approx  (d +1 )  \| \nabla F(\theta_t) \|^2  $. We refer to $\text{Var}_{\widehat{P}}^M$ as the "main component" of the variance $\text{Var}_{\widehat{P}}$. Similarly we define $\text{Var}_{ES}^M =(d +1 )  \| \nabla F(\theta_t) \|^2$ and use the same name, "main component" of the variance $\text{Var}_{ES}^M$.

The optimal $p^{t}$, that which minimizes $\text{Var}^M_{\widehat{P}}$ equals:

\begin{equation*}
    (p^{t})^* = \frac{ \|\left(  \nabla F(\theta_t)     \right)_{active} \| \sqrt{ d_{\mathrm{active}} + 2 } }{  \| \left(\nabla F(\theta_t) \right)_{active} \| \sqrt{d_{\mathrm{active}} + 2} + \| \left( \nabla F(\theta)    \right)_\perp \| \sqrt{d_{\perp} + 2}  }  
\end{equation*}



\begin{proof}
Roughly the same argument as above yields the desired result.
\end{proof}

\begin{lemma}\label{lemma::optimal_variance}
The optimal variance $\text{Var}^M_{\widehat{P}^*}$ corresponding to $(p^{t})^*$ equals:

\begin{equation}
     \left[ \| \left( \nabla F(\theta_t)       \right)_{active} \| \sqrt{d_{\mathrm{active}} + 2} +     \|\left(    \nabla F(\theta_t) \right)_{\perp} \|\sqrt{d_{\perp} + 2} \right]^2 - \| \nabla F(\theta_t) \|^2 
\end{equation}

\end{lemma}

\begin{proof}
The statement follows directly from substituting the expression for $(p^{t})^*$ into the variance formula.
\end{proof}

\begin{theorem}\label{theorem::condition_for_domination}
$\text{Var}_{\widehat{P}^*}^M \leq \text{Var}_{\text{ES}}^M$ if
\begin{equation*}
    | \sqrt{d_{\mathrm{active}} + 2}\| \left( \nabla F(\theta_t) \right)_{\perp} \| -  \sqrt{d_{\perp} + 2} \| \nabla F(\theta_t)_{active} \| |\geq \sqrt{2} \| \nabla F(\theta_t)\|
\end{equation*}

\end{theorem}

\begin{proof}
By definition, $\text{Var}_{\widehat{P}^*}^M < \text{Var}_{\text{ES}}^M$ if:

\begin{equation}\label{equation::variance_comparison1}
     \left(\| \left( \nabla F(\theta_t)       \right)_{active} \| \sqrt{d_{\mathrm{active}} + 2} +     \|\left(    \nabla F(\theta_t) \right)_{\perp} \|\sqrt{d_{\perp} + 2} \right)^2 < \| \nabla F(\theta_t) \|^2 (d+2)
\end{equation}

Let $a_1 = \sqrt{d_{\mathrm{active}}+2}, a_2 = \sqrt{d_{\perp} + 2}$, $b_1 = \| \left( \nabla F(\theta_t)\right)_{active}\|,b_2 = \|\left( \nabla F(\theta_t) \right)_{\perp}\|$, $a = \sqrt{d+2}$ and $b = \|\nabla F(x_t)\|$. 

The following relationships hold: $b_1^2 + b_2^2 = b^2$ and $a_1^2 + a_2^2 - 2 = a^2$. The bound we want to prove in Equation \ref{equation::variance_comparison1} reduces to finding conditions under which:

\begin{equation*}
    (a_1b_1 + a_2b_2)^2 \leq (b_1^2 + b_2^2) (a_1^2 + a_2^2 - 2)
\end{equation*}

Which holds iff:

\begin{equation*}
    2b_1^2 + 2b_2^2 \leq a_2^2b_1^2 +  a_1^2b_2^2 - 2a_1a_2b_1b_2
\end{equation*}

The later holds iff:

\begin{equation*}
    |a_1b_2 - a_2b_1| \geq \sqrt{2}b
\end{equation*}

Which holds iff:

\begin{equation*}
    \left |\sqrt{d_{\mathrm{active}} +2} \| \left(\nabla F(x_t) \right)_{\perp} \| - \sqrt{d_{\perp} + 2} \| \left(  \nabla F(x_t) \right)_{active} \| \right |\geq \sqrt{2}\| \nabla F(x_t) \|
\end{equation*}

\end{proof}

The inequality is strict for example when $\| \left(   \nabla F(x_t) \right)_{\perp} \|  = 0$ and $d_{\perp}  \geq 1$.

This in turn implies that, after taking into account the bias terms:

\begin{theorem}\label{theorem::condensed_variance_results}
If $\epsilon > 0$. If $\sigma < \frac{1}{35}\sqrt{\frac{\epsilon \min((p^{t})^*, 1-(p^{t})^*)}{\tau  d^3 \max(L, 1)} }$,
we denote by $\text{Var}_{(\widehat{P})^*}$ as the variance of the gradient estimator $z_1$ corresponding to the optimal (for $\text{Var}_{\widehat{P}}^M$) probability $(p^{t})^*$ and 
\begin{equation*}
    | \sqrt{d_{\mathrm{active}} + 2}\| \left( \nabla F(\theta_t) \right)_{\perp} \| -  \sqrt{d_{\perp} + 2} \| \nabla F(\theta_t)_{active} \| |\geq \sqrt{2} \| \nabla F(\theta_t)\|
\end{equation*}
Then:
\begin{equation*} 
\text{Var}_{\widehat{P}^*} \leq \text{Var}_{ES} + \epsilon
\end{equation*}
\end{theorem}

\subsection{Adaptive Mirror Descent for variance reduction.} 

In this section we propose an adaptive procedure to learn the optimal probability parameter $(p^t)^*$ (as introduced in the previous section) this is necessary since as it can be infered from the discussion in section \ref{section::variance_reduction_non_isotropic}, the optimal variance depends of unknown parameters such as the projection of the true gradient onto the subspaces spanned by $\mathbf{U}$ and $\mathbf{U}^\perp$. The final result of this section \ref{theorem::regret_mirror_descent_appendix} corresponds to Theorem \ref{theorem::regret_mirror_descent} in the main section of the text.

Let $\mathbf{p}^l = \binom{p^l}{1-p^l}$. The main component $\Gamma$ of the variance of ${\widehat{\nabla}^{\mathrm{AT},\mathrm{asebo}}_{\mathrm{MC},k=1}F_{\sigma}(\theta)}$ as a function of $\mathbf{p}^l$ equals (Lemma \ref{theorem:combined_theorem_variance1}) :
\begin{equation}\label{equation::gamma_definition}
\Gamma = \ell(\mathbf{p}^l) =\frac{d_{\mathrm{active}} + 2}{\mathbf{p}^l(1)} s_{\mathbf{U}^{\mathrm{act}}} + \frac{ d_{\perp} +2}{\mathbf{p}^l(2)} s_{\mathbf{U}^{\perp}} - \|\nabla F(\theta) \|^2.
\end{equation}

In order to avoid the gradients to explote in norm, we parametrise $\mathbb{p}^l$ as follows:

\begin{equation*}
    \mathbf{p}^l = (1-2\beta) \mathbf{q}^l + \binom{\beta}{\beta}
\end{equation*}

For $\mathbf{q}^l \in \Delta_2$ and $\beta \in (0,1)$, the boundary probability bias.

Notice that $\Gamma$ is a convex function of $\mathbf{p}$ and also a convex function of $\mathbf{q}$. With a slight abuse of notation we denote $\ell(\mathbf{q}^l)$ as the loss parametrized by $\mathbf{q}^l$ (which satisfies $\ell( \mathbf{q}^l) = \ell(\mathbf{p}^l)$).


The gradient $\nabla_{\mathbf{q}^l} \ell(\mathbf{q}^l)$ equals:
\begin{equation*}
    \nabla_{\mathbf{q}^l}  \ell(\mathbf{q}^{l}) = (1-2\beta) \binom{-\frac{d_{\mathrm{active}} + 2}{((1-2\beta)\mathbf{q}^{l}(1) + \beta)^2} s_{\mathbf{U}^{\mathrm{ort}}} }{ -\frac{d_{\perp} + 2}{((1-2\beta)\mathbf{q}^{l}(2) + \beta)^2} s_{\mathbf{U}^{\perp}}} ,
\end{equation*}

And can be approximated (at the cost of some bias) using function evaluations.

\begin{lemma}\label{lemma::approximation_stochastic_gradient}
The gradient $\nabla_{\mathbf{q}^l} \ell(\mathbf{q}^l)$ satisfies:
\begin{equation*}
  \left \| \nabla_{\mathbf{q}^l}  \ell(\mathbf{q}^{l}) -
    \mathbb{E}\left[(1-2\beta) \binom{-\frac{a_l(d_{\mathrm{active}} + 2)}{((1-2\beta)\mathbf{p}^{l}(1) + \beta)^3}  }{ -\frac{(1-a_l)(d_{\perp} + 2)}{((1-2\beta)\mathbf{p}^{l}(2) + \beta)^3} } v_l^2    \right] \right \| \leq \frac{\epsilon (d+2)}{\left(\min(\mathbf{p}^l(1), \mathbf{p}^l(2))\right)^3} \leq \frac{\epsilon (d+2)}{\beta^3} ,
\end{equation*}
where $v_{l} = \frac{1}{2\sigma}\left(  F(\theta + \mathbf{g}_l) - F(\theta - \mathbf{g}_l)  \right) $ .
\end{lemma}

\begin{proof}
We start with some notation borrowed from the previous section:
 \begin{equation*}
     \xi_{\mathbf{g}_l}^{(2)}(\theta ) =  \underbrace{\left(  \frac{F(\theta + \sigma \mathbf{g}_l) - F(\theta - \sigma \mathbf{g}_l) }{2\sigma}   \right)^2}_{v_l^2} - \left(  \mathbf{g}_l^\top \nabla F(\theta_t)  \right)^2 
 \end{equation*}
 Observe that:
\begin{align*} 
 |\xi^{(2)}_{\mathbf{g}_l}(\theta ) |&= \left| \left(  \frac{F(\theta + \sigma \mathbf{g}_l) - F(\theta - \sigma \mathbf{g}_l) }{2\sigma}   \right)^2 - \left(  \mathbf{g}_l^\top \nabla F(\theta)  \right)^2 \right| \\
 &\leq \xi_{\mathbf{g}_l}(\theta)^2 + 2 \left|  \mathbf{g}_l^\top \nabla F(\theta) \xi_{\mathbf{g}_l}(\theta)  \right|\\
 &\leq \sigma^4 \tau^2 \| \mathbf{g}_l \|^6 + 2 \sigma^2 \tau L \|\mathbf{g}_l \|^4
 \end{align*}

Since $\sigma < \frac{1}{35}\sqrt{\frac{\epsilon \min(p^{t}, 1-p^{t})}{\tau  d^3 \max(L, 1)} } $:
\begin{equation*}
\mathbb{E}\left[ |\xi_{\mathbf{g}_l}^{(2)}(\theta) |\right] \leq \epsilon
\end{equation*}

The result follows.

\end{proof}



Let $p^l = (1-2\beta)q^l + \beta$ be the probability that we choose to sample from the subspace $\mathcal{L}^{\mathrm{ES}}_{\mathrm{active}}$ and $1-p^l$ the probability that we choose to sample from $\mathcal{L}^{\mathrm{ES}, \perp}_{\mathrm{active}}$. Let $a_l$ be a Bernoulli random variable $a_l \in \{ 0, 1\}$ with $\mathbb{E}\left[ \binom{a_l}{1-a_l}  \right] = \mathbf{p}^l$. Define the stochastic gradient (with respect to $\mathbf{q}^l$):
\begin{equation*}
     \mathbf{e}_l = (1-2\beta)\begin{bmatrix} \left(-\frac{a_l(d_{\mathrm{active}}+2)}{p^{3}_l} \right) \\
\left(-\frac{(1-a_l)(d_{\perp} +2)}{(1-p_l)^3} \right) \end{bmatrix} v_l^2 
\end{equation*}
By definition this random vector (conditioned on the choice of $\mathbf{p}^l$) satisfies:
\begin{equation*}
  \left \| \mathbb{E}\left[    e_l \right] - \nabla_{\mathbf{q}^l} \ell(\mathbf{q}^l) \right \| \leq \frac{\epsilon (d+2)}{\left(\min(\mathbf{p}^l(1), \mathbf{p}^l(2))\right)^3} \leq \frac{\epsilon (d+2)}{\beta^3} ,
\end{equation*}

If $\epsilon$ is chosen small enough, the bias can be driven to be arbitrarily small.

\subsubsection{Mirror descent}

We treat this problem as that of minimizing the loss $\ell$ over the two dimensional simplex and resort to adapt a version of Mirror descent for it. As opposed the case of projected gradient descent, mirror descent performs updates that are adapted to the geometry of the simplex, ensuring the iterates always belong to the simplex and no projection step is necessary. The mirror descent updates are:
\begin{align*}
\mathbf{q}_{l}(1) &= \frac{\mathbf{q}_{l-1}(1) \exp(-\alpha\mathbf{e}_l(1))}{ \mathbf{q}_{l-1}(1) \exp(-\alpha \mathbf{e}_l(1))   + (\mathbf{q}_{l-1}(2)) \exp(-\alpha \mathbf{e}_l(2))} \\
\mathbf{q}_{l}(2) &= \frac{\mathbf{q}_{l-1}(2) \exp(-\alpha \mathbf{e}_l(2))}{ \mathbf{q}_{l-1}(1) \exp(-\alpha \mathbf{e}_l(1))   + (\mathbf{q}_{l-1}(2)) \exp(-\alpha \mathbf{e}_l(2))} 
\end{align*}

For a step size parameter $\alpha$. 






\subsection{Regret guarantees}

Using he notation in  \url{https://www.stat.berkeley.edu/~bartlett/courses/2014fall-cs294stat260/lectures/mirror-descent-notes.pdf}, In this case let $R(\mathbf{q}) = \mathbf{q}(1) \log(\mathbf{q}(1)) + \mathbf{q}(2) \log(\mathbf{q}(2)) - \mathbf{q}(1) - \mathbf{q}(2)$ and therefore:

\begin{equation}
    \nabla R(\mathbf{q}) = \binom{\log(\mathbf{q}(1) ) }{\log( \mathbf{q}(2))}
\end{equation}

The Fenchel conjugate of $R$ equals:

\begin{equation}
    R^*(\mathbf{q}) = e^{\mathbf{q}(1)} + e^{\mathbf{q}(2)}
\end{equation}

And therefore the gradient of the Fenchel conjugate equals:

\begin{equation}
    \nabla R^*(\mathbf{q}) = \binom{ \exp( \mathbf{q}(1)  ) }{\exp(\mathbf{q}(2) )  } 
\end{equation}

And:

\begin{equation}
    D_R(\mathbf{q}_1, \mathbf{q}_2) =  \mathbf{q}_1(1) \log \left( \frac{\mathbf{q}_1(1)}{\mathbf{q}_2(1) } \right) +  \mathbf{q}_1(2) \log\left( \frac{\mathbf{q}_1(2)}{\mathbf{q}_2(2) }\right ) + \mathbf{q}_2(1) - \mathbf{q}_1(1) + \mathbf{q}_2(2) - \mathbf{q}_1(2)
\end{equation}

Recall the update behind Mirror descent takes the form (stepsize $\alpha$:

\begin{enumerate}
    \item Play $\binom{a_l}{1-a_l}$ such that $\mathbb{E}[\binom{a_l}{1-a_l}] = \mathbf{p}^l$.
    \item Let $w_{l+1} = \nabla R^*\left(   \nabla R(\mathbf{p}^l) - \alpha \mathbf{e}_l \right)$
    \item Let $\mathbf{p}^{l+1} = \arg\min_{  \mathbf{p} \in \Delta_2 }       D_R(\mathbf{p}, w_{t+1}   )$
\end{enumerate}

Recall the general definition of Bregman divergence:

\begin{equation}
    D_\Psi(u, v) = \Psi(u) - \Psi(v) - \langle \nabla \Psi(v), u-v \rangle
\end{equation}

The following regret guarantee holds for Mirror descent (see \url{https://www.stat.berkeley.edu/~bartlett/courses/2014fall-cs294stat260/lectures/mirror-descent-notes.pdf}):

\begin{theorem}
If at time $l$ a convex loss function $f_l$ is revealed to the player and the player performs the mirror descent step using $\nabla f_l$ as a proxy linear function, with actions (from the mirror descent step) $a_l$ at time $l$, for any $a$ in the intersection of all of $f_l$'s domains, the following regret bound holds:

\begin{align}
\sum_{l=1}^C \left( f_l(a_l) - f_l(a)    \right) &\leq  \sum_{l=1}^C \nabla f_l(a_l)^\top (a_l - a) \\
&\leq \frac{1}{\alpha} \left( R(a) - R(a_1) +\sum_{l=1}^C D_{R^*}( \nabla R(a_l) - \alpha \nabla f_l(a_l), \nabla R(a_l)  )      \right)
\end{align}
\end{theorem}

Also remember that if $R^*$ is $\theta$-smooth with respect to some norm $\| \cdot \|$, we can upper bound $D_{R^*}$. The former ($R^*$ being $\theta-$smooth) holds if $R$ is $\frac{1}{\theta}$-strongly convex with respect to the dual norm $\| \cdot \|_*$. When $R$ equals the entropy, this is $1-$strongly convex with respect to the $L_1$ norm and hence $R^*$ is $1-$strongly smooth with respect to the $L_\infty$ norm:

\begin{equation}
D_{R^*}(a, b) \leq \frac{\| a-b\|_\infty^2}{2}
\end{equation}

In our case, let $f_l(\mathbf{q}) = \mathbf{e}_l^\top \mathbf{q}$. Using the upper bound previously described for $D_R$. For any $\mathbf{q} \in \Delta_2$:

\begin{equation*}
    \sum_{l=1}^C f_l(\mathbf{q}^l) - f_l(\mathbf{q}) \leq \frac{1}{\alpha} \left(  R(\mathbf{q}) -    R(\mathbf{q}^1) + \alpha^2 \sum_{l=1}^C \frac{\| \nabla f_l( \mathbf{q}_l  ) \|_\infty^2}{2}    \right)
\end{equation*}

Taking expectations, since $\left | \mathbb{E}[ f_l(\mathbf{q}) | \mathbf{q}_l] - \nabla_{\mathbf{q}^l}^\top \ell(\mathbf{q}^l) \mathbf{q} \right |\leq \frac{\epsilon (d+2)}{\beta^3}$ we obtain the following result:

\begin{equation*}
   \left(  \sum_{l=1}^C   \nabla_{\mathbf{q}^l}^\top \ell(\mathbf{q}^l) \left(     \mathbf{q}^l       -  \mathbf{q}  \right) \right) -C\frac{2\epsilon (d+2)}{\beta^3} \leq \frac{1}{\alpha} \left(  R(\mathbf{q}) -   \mathbb{E}\left[   R(\mathbf{q}^1) \right] + \alpha^2 \sum_{l=1}^C \frac{\mathbb{E}\left[  \|  \mathbf{e}_l \|_\infty^2\right]}{2}    \right)
\end{equation*}

Now we bound the Right Hand side of the expression above. Notice that $R(\mathbf{q}) \leq 2$ and $R(\mathbf{q}_1) \geq 0$. We can also bound the expectation $\mathbb{E}\left[   \| \mathbf{e}_l \|_\infty^2\right]$.
\begin{lemma}

$\left| \mathbb{E}\left[   \| \mathbf{e}_l \|^2\right] - \| \nabla_{\mathbf{q}^l}  \ell(\mathbf{q}^{l})\|^2  \right| \leq \frac{ \epsilon (d+1)}{\beta^3}$

\end{lemma}

\begin{proof}
A similar calculation as in Lemma \ref{lemma::approximation_stochastic_gradient} yields the desired result.
\end{proof}

Since:
\begin{equation}
    \mathbb{E}\left[   \| \mathbf{e}_l \|_\infty^2\right] \leq \mathbb{E}\left[   \| \mathbf{e}_l \|^2\right] 
\end{equation}

And $\| \nabla_{\mathbf{q}^l}  \ell(\mathbf{q}^{l})\|^2 \leq \frac{1}{\beta^4}\left(   (d_{\mathrm{active}}+2)^2 s_{\mathbf{U}^{\mathrm{ort}}}^2+(d_{\mathrm{\perp}}+2)^2 s_{\mathbf{U}^{\perp}}^2\right)$. 

We obtain the following bound:
\begin{equation*}
   \left(  \sum_{l=1}^C   \nabla_{\mathbf{q}^l}^\top \ell(\mathbf{q}^l) \left(     \mathbf{q}^l       -  \mathbf{q}  \right) \right) -C\frac{2\epsilon (d+2)}{\beta^3} \leq \frac{2}{\alpha}   +  \frac{\alpha C }{2\beta^4}\left(   (d_{\mathrm{active}}+2)^2 s_{\mathbf{U}^{\mathrm{ort}}}^2+(d_{\mathrm{\perp}}+2)^2 s_{\mathbf{U}^{\perp}}^2\right) +  \frac{ \alpha C \epsilon (d+1)}{\beta^3}
\end{equation*}

The following theorem follows:

\begin{theorem}
If $\alpha = \frac{2\beta^2}{\sqrt{C}\sqrt{   (d_{\mathrm{active}}+2)^2 s_{\mathbf{U}^{\mathrm{ort}}}^2+(d_{\mathrm{\perp}}+2)s_{\mathbf{U}^{\perp}}^2 }}$ and $\epsilon = \frac{\beta^3}{2C (d+1)}$, for any $\mathbf{q} \in \Delta_2$:

\begin{equation*}
    \left(  \sum_{l=1}^C   \nabla_{\mathbf{q}^l}^\top \ell(\mathbf{q}^l) \left(     \mathbf{q}^l       -  \mathbf{q}  \right) \right) \leq \frac{\sqrt{C} \sqrt{   (d_{\mathrm{active}}+2)^2 s_{\mathbf{U}^{\mathrm{ort}}}^2+(d_{\mathrm{\perp}}+2)s_{\mathbf{U}^{\perp}}^2 } }{\beta^2}+ 1
\end{equation*}

\end{theorem}

\begin{proof}
Plugging in this value of $\alpha$:

\begin{align*}
    \left(  \sum_{l=1}^C   \nabla_{\mathbf{q}^l}^\top \ell(\mathbf{q}^l) \left(     \mathbf{q}^l       -  \mathbf{q}  \right) \right) &\leq \frac{\sqrt{C} \sqrt{   (d_{\mathrm{active}}+2)^2 s_{\mathbf{U}^{\mathrm{ort}}}^2+(d_{\mathrm{\perp}}+2)s_{\mathbf{U}^{\perp}}^2 } }{\beta^2} \\
    &+ \left(1+\frac{2\beta^2}{\sqrt{C}\sqrt{   (d_{\mathrm{active}}+2)^2 s_{\mathbf{U}^{\mathrm{ort}}}^2+(d_{\mathrm{\perp}}+2)s_{\mathbf{U}^{\perp}}^2 }} \right)\frac{C\epsilon(d+1) }{\beta^3}
\end{align*}

By setting $\epsilon = \frac{\beta^3}{2C (d+1)}$ the result follows. Assuming $C$ is large enough so that $\alpha < 1$.
\end{proof}

Since $\ell(\mathbf{q})$ is a convex function of $\mathbf{q}$ for all $l$ and $\mathbf{q} \in \Delta_2$:

\begin{equation*}
    \ell(\mathbf{q}^l  ) - \ell(\mathbf{q})  \leq \nabla_{\mathbf{q}^l}^\top \ell(\mathbf{q}^l) \left(     \mathbf{q}^l       -  \mathbf{q}  \right)  
\end{equation*}

Which in turn implies the main result of this section:

 \begin{theorem}\label{theorem::regret_mirror_descent_appendix}
If $\alpha = \frac{2\beta^2}{\sqrt{C}\sqrt{   (d_{\mathrm{active}}+2)^2 s_{\mathbf{U}^{\mathrm{ort}}}^2+(d_{\mathrm{\perp}}+2)s_{\mathbf{U}^{\perp}}^2 }}$ and $\epsilon = \frac{\beta^3}{2C (d+1)}$, for any $\mathbf{q} \in \Delta_2$:

\begin{equation*}
    \mathbb{E}\left[  \sum_{l=1}^C   \ell(\mathbf{q}^l  ) - \ell(\mathbf{q}) \right] \leq \frac{\sqrt{C} \sqrt{   (d_{\mathrm{active}}+2)^2 s_{\mathbf{U}^{\mathrm{ort}}}^2+(d_{\mathrm{\perp}}+2)s_{\mathbf{U}^{\perp}}^2 } }{\beta^2}+ 1
\end{equation*}

This is equivalent to the result stated in the main paper.

\end{theorem}

\section{Additional Implementation Details}

In this section we present additional details on our experimental results, for both the RL tasks and $\mathrm{Nevergrad}$ functions.

\subsection{Reinforcement Learning Experiment Details}

We provide additional details regarding the RL experiments below.

\paragraph{State Normalization.} State-of-the-art policy optimization baselines such as $\mathrm{PPO}$/$\mathrm{TRPO}$ \citep{baselines} and 
the original $\mathrm{ARS}$ \citep{horia} apply state normalization as part of the implementation. In particular, the algorithms 
maintain a component-wise running average of mean $\bar{s}$ and standard deviation vector $\sigma(s)$ of the state. When at given state $s_t$, 
the algorithm computes the normalized state $\tilde{s}_t = \frac{s_t - \bar{s}}{\sigma(s)}$ before inputing to the policy network to 
compute actions $a_t = \pi(\tilde{s}_t)$. For $\mathrm{PPO}$/$\mathrm{TRPO}$, since the optimization is based on back-propagation of neural networks, 
properly scaling the inputs $s_t \rightarrow \tilde{s}_t$ is critical for the performance. In all experiments, we remove state normalization 
mechanism from the implementation to test the robustness of various blackbox optimization algorithms. 
Notice that as reported by \cite{horia}, state normalization was not needed in $\mathrm{ARS}$ to learn good policies for RL tasks under consideration in this paper.
As a result, we observe that $\mathrm{PPO}$/$\mathrm{TRPO}$ 
underperform other ES algorithms for most tasks.

\paragraph{Benchmark Environments.} Benchmark environments are from OpenAI gym \citep{Gym}. These environments have variable sizes of observation space and action space: $\mathrm{Swimmer}$-$\mathrm{v2}$ $|\mathcal{S}|=8$, $|\mathcal{A}|=2$; $\mathrm{Hopper}$-$\mathrm{v2}$ $|\mathcal{S}|=11$, $|\mathcal{A}|=3$; $\mathrm{HalfCheetah}$-$\mathrm{v2}$ $|\mathcal{S}|=17$, $|\mathcal{A}|=6$; $\mathrm{Thrower}$-$\mathrm{v2}$ $|\mathcal{S}|=23$, $|\mathcal{A}|=7$; $\mathrm{Pusher}$-$\mathrm{v2}$ $|\mathcal{S}|=23$, $|\mathcal{A}|=7$; $\mathrm{Walker2d}$-$\mathrm{v2}$ $|\mathcal{S}|=17$, $|\mathcal{A}|=6$;  $\mathrm{Reacher}$-$\mathrm{v2}$ $|\mathcal{S}|=11$, $|\mathcal{A}|=2$.  All environments have a natural termination condition specified in the simulation environment. 

\paragraph{Policy Architecture.} All baseline algorithms involve training a parameterized policy $\pi_\theta(a|s)$ using sample gradient estimates generated from the environment. The policy architecture is shared across all algorithms: a 2-layer feed-forward neural network with $\text{tanh}$ non-linearity and $h$ hidden units per layer. The input to the network is the state $s \in \mathcal{S}$. For all ES-based algorithms (Vanilla ES, $\mathrm{CMA}$-$\mathrm{ES}$, $\mathrm{ARS}$ and $\mathrm{ASEBO}$), the output of the network is the action $a_\theta(s) \in \mathcal{A}$. For policy optimization algorithms ($\mathrm{PPO}$, $\mathrm{TRPO}$), the output of the network is a mean of Gaussian $\mu_\theta(s)$ and we draw actions from a factorized Gaussian distribution $a \sim \mathcal{N}(\mu_\theta(s),\sigma^2 \mathbb{I})$ where we separately parameterize a standard deviation parameter $\sigma$ shared across dimensions. The sizes of hidden layers where: $4$ for $\mathrm{LQR}$, $16$ for $\mathrm{Swimmer}$-$\mathrm{v2}$, $\mathrm{Hopper}$-$\mathrm{v2}$ and $\mathrm{Reacher}$-$\mathrm{v2}$, $32$ for $\mathrm{HalfCheetah}$-$\mathrm{v2}$ and $\mathrm{Walker2d}$-$\mathrm{v2}$, reflecting the difficulty of each task.

\paragraph{Optimization.} Our method ($\mathrm{ASEBO}$) and most baselines (Vanilla ES, $\mathrm{ARS}$, $\mathrm{CMA}$-$\mathrm{ES}$ and $\mathrm{PPO}$) apply SGD based methods and we apply the $\mathrm{Adam}$ optimizer to stabilize the gradients.  

\subsubsection{Baseline Algorithms}

\paragraph{Vanilla ES.} Vanilla ES is the simplest evolutionary algorithm applied in RL tasks \citep{stockholm,ES}. We apply the antithetic sampling scheme as applied in \citep{stockholm}. Our implementation does not rank the rewards as in \citep{ES}, and as previously discussed does not include observation normalization.

\paragraph{CMA-ES variants.} Covariance Matrix Adaptation Evolution Strategy is a state-of-the-art and popular black box optimization algorithm \citep{hansen1996adapting}. $\mathrm{VkD}$-$\mathrm{CMA}$-$\mathrm{ES}$ and $\mathrm{LM}$-$\mathrm{MA}$-$\mathrm{ES}$ are recently proposed variant designed for high dimensional blackbox functions. For $\mathrm{VkD}$-$\mathrm{CMA}$-$\mathrm{ES}$ we use the open source implementation from \emph{pycma} available at \url{http://github.com/CMA-ES/pycma}. We use the default hyper-parameters in the original code base with the standard deviation parameter $\sigma = 1.0$. For $\mathrm{LM}$-$\mathrm{MA}$-$\mathrm{ES}$ we use the implementation from \cite{challenges_esrl}.

\paragraph{ARS.} Augmented Random Search \citep{horia} is based on the code released by the original paper. We use the standard deviation $\sigma = 0.02$ and learning rate $\eta = 0.01$. The hyper-parameters are tuned on top of the default hyper-parameters in the original code base. We remove the observation normalization utility in the original code for fair comparison.

\paragraph{ASEBO.} We propose Adaptive Sample Efficient Blackbox Optimization in this work. Our algorithms have the following hyper-parameters: the covariance decay parameter $\lambda =0.995$ (slow decay), proportion of variance of the active (PCA) space $\epsilon = 0.995$, standard deviation parameter $\sigma = 0.02$. We set the learning rate $\eta = 0.02$. 

\paragraph{Trust Region Policy Optimization.} Trust Region Policy Optimization ($\mathrm{TRPO}$) is based on the implementation of OpenAI baseline \citep{baselines}. We use the default training hyper-parameters in the code base: we collect $N=1024$ samples per batch to compute a policy gradient, with the trust region size parameter $\epsilon = 0.01$. We remove the observation normalization utility in the original code for fair comparison.

\paragraph{Proximal Policy Optimization.} Proximal Policy Optimization ($\mathrm{PPO}$) \citep{schulman2017proximal} is also based on the implementation of OpenAI baseline \citep{baselines}. We use the default hyper-parameters in the code base: we collect $N=2048$ samples per batch to compute policy gradients and set the clipping coefficient $\epsilon=0.2$. The learning rate is set to be $\alpha=3\cdot 10^{-5}$ for all environments. We remove the observation normalization utility in the original code for fair comparison.

\subsection{$\mathrm{Nevergrad}$ Experiment Details}

\paragraph{Function Settings.} We tested the following functions: $\mathrm{cigar}$, $\mathrm{ellipsoid}$, $\mathrm{sphere}$, $\mathrm{sphere2}$, $\mathrm{rosenbrock}$, $\mathrm{rastrigin}$ and $\mathrm{lunacek}$. In each case we used $d=1000$ to evaluate $\mathrm{ASEBO}$ in a high dimensional setting.

\paragraph{Algorithm Hyper-Parameters.} We use the same hyper-parameters across all functions. For $\mathrm{ASEBO}$ and $\mathrm{Vanilla} \mathrm{ES}$, we use $\eta = 0.02$. For $\mathrm{ASEBO}$ we set $\lambda = 0.99$. For $\mathrm{VkD}$-$\mathrm{CMA}$-$\mathrm{ES}$ we use the default parameters from the $\mathrm{pycma}$ package, and for $\mathrm{LM}$-$\mathrm{MA}$-$\mathrm{ES}$ we use the implementation from \cite{challenges_esrl}.

\section{The Algorithm - Additional Details \& Analysis}

We provide here few variations of the $\mathrm{ASEBO}$ algorithm from the main body of the paper, namely:

\begin{algorithm}[H]
\caption{$\mathrm{ASEBO}$ Algorithm - extended version}
\textbf{Hyperparameters:}  number of iterations of full sampling $l$, smoothing parameter $\sigma>0$,
step size $\eta$, PCA threshold $\epsilon$, decay rate $\gamma$, 
total number of iterations $T$.\; \\
\textbf{Input:}  blackbox function $F$, vector $\theta_0 \in \mathbb{R}^{d}$ where optimization starts. 
                 $\mathrm{Cov}_0 \in \{0\}^{d \times d}$, $p^{0}=0$.\; \\
\textbf{Output:} vector $\theta_{T}$. \; \\
\For{$t=0, \ldots, T-1$}{
  \If {$t < l$}{ 
    Take $n_t = d$. Sample $\mathbf{g}_1, \cdots, \mathbf{g}_{n_{t}}$ from $\mathcal{N}(0, \mathbf{I}_{d})$ (independently). \; 
  }
  \Else {
    1. Take top $r$ eigenvalues $\lambda_{i}$ of $\mathrm{Cov}_{t}$, where $r$ is smallest such that: $\sum_{i=1}^{r}\lambda_{i} \geq \epsilon \sum_{i=1}^{d} \lambda_{i}$, 
       using its $\mathrm{SVD}$ as described in text and take $n_{t}=r$.\; \\
    2. Take the corresponding eigenvectors $\mathbf{u}_{1},...,\mathbf{u}_{r} \in \mathbb{R}^{d}$ and let $\mathbf{U} \in \mathbb{R}^{d \times r}$
       be obtained by stacking them together. Let $\mathbf{U}^{\mathrm{act}} \in \mathbb{R}^{d \times r}$ be obtained from stacking together 
       some orthonormal basis of $\mathcal{L}^{\mathrm{ES}}_{\mathrm{active}}   
       \overset{\mathrm{def}}{=} \mathrm{span}\{\mathbf{u}_{1},...,\mathbf{u}_{r}\}$.
       Let $\mathbf{U}^{\perp} \in \mathbb{R}^{d \times (d-r)}$ be obtained from
       stacking together some orthonormal basis of the orthogonal complement $\mathcal{L}^{\mathrm{ES}, \perp}_{\mathrm{active}}$ of $\mathcal{L}^{\mathrm{ES}}_{\mathrm{active}}$.       
       \\
    3. Sample $\mathbf{g}_{1},...,\mathbf{g}_{n_{t}}$ from  $\mathcal{N}(0,\sigma \Sigma)$ (independently), 
       where $\Sigma = \frac{1-p^{t}}{d}\mathbf{I}_{d} + \frac{p^{t}}{r}\mathbf{U}\mathbf{U}^{\top}$ ($\textbf{V0}$) or 
       sample $n_{t}$ vectors $\mathbf{g}_{1},...,\mathbf{g}_{n_{t}}$ as follows: with probability $1-p^{t}$ from $\mathcal{N}(0,\mathbf{U}^{\perp}(\mathbf{U}^{\perp})^{\top})$
       and with probability $p^{t}$ from  $\mathcal{N}(0,\mathbf{U}^{\mathrm{act}}(\mathbf{U}^{\mathrm{act}})^{\top})$ ($\textbf{V1}$). \\
    4. Renormalize $\mathbf{g}_{1},...,\mathbf{g}_{n_{t}}$ such that marginal distributions $\|\mathbf{g}_{i}\|_{2}$ are  $\chi(d)$.
  }
  1. Compute $\widehat{\nabla}_{\mathrm{MC}}^{\mathrm{AT}}F(\theta_{t})$ as:
  \begin{align*}
    \widehat{\nabla}_{\mathrm{MC}}^{\mathrm{AT}}F(\theta_{t}) =  \frac{1}{2n_{t}\sigma} \sum_{j=1}^{n_{t}} (F(\theta_t +  \mathbf{g}_j) - F(\theta_t - \mathbf{g}_j)) \mathbf{g}_j.
  \end{align*}  
  2. Set $\mathrm{Cov}_{t+1} = \lambda \mathrm{Cov}_{t} + (1-\lambda) \Gamma$, where 
     $\Gamma = \widehat\nabla^{\mathrm{AT}}_{\mathrm{MC}}F_{\sigma}(\theta_{t}) (\widehat\nabla^{\mathrm{AT}}_{\mathrm{MC}}F_{\sigma}(\theta_{t}))^{\top}$.
     \\
  3. Set $p^{t+1} = p_{\mathrm{opt}}$ for $p_{\mathrm{opt}}$ output by
      Algorithm 2 (from the main body) or $p^{t+1} = \frac{\widehat{r}}{\widehat{r}+1}$, where:
   \begin{align*} 
   \widehat{r} = \frac{\|(\widehat\nabla F_{\sigma}(\theta_{t}))_{\mathrm{active}}\|_{2}}{\|(\widehat\nabla F_{\sigma}(\theta_{t}))_{\perp}\|_{2}},
   \end{align*} 
   is computed by Algorithm 4 (see: below)
   and scalars $\|(\widehat\nabla F_{\sigma}(\theta_{t}))_{\mathrm{active}}\|_{2}$, $\|(\widehat\nabla F_{\sigma}(\theta_{t}))_{\perp}\|_{2}$ stand for the estimates of
   $\|(\nabla F_{\sigma}(\theta_{t}))_{\mathrm{active}}\|_{2}$ and $\|(\nabla F_{\sigma}(\theta_{t}))_{\perp}\|_{2}$. \\
  4. Set $\theta_{t+1} \leftarrow \theta_{t} + \eta \widehat{\nabla}_{\mathrm{MC}}^{\mathrm{AT}}F(\theta_{t})$. 
 }
\label{Alg:asebo}
\end{algorithm}

\begin{itemize}
    \item we propose one more method for sampling from heterogeneous distributions (see: version \textbf{V0} in Algorithm 3; the default one that we present in the main body is called \textbf{V1} here),
    \item we propose to use compressed sensing techniques (Algorithm 4) as an alternative to the contextual bandits method from the main body (Algorithm 2); the bandits method can be seen as an extension of the compressed sensing techniques.
\end{itemize}

\begin{algorithm}[H]
\textbf{Hyperparameters:} smoothing parameter $\sigma$, horizon $C$. \\
\textbf{Input:} subspaces: $\mathcal{L}^{\mathrm{ES}}_{\mathrm{active}}$,
$\mathcal{L}_{\mathrm{active}}^{\mathrm{ES},\perp}$, function $F$, vector $\theta_{t}$. \\
\textbf{Output:} ratio $\widehat{r}$. \\
1. Initialize square norm averages $s^{\mathrm{active}}_0=s^{\perp}_0 = 0$. \newline
\For{$l = 1, \cdots , C$ }{
    1. Sample $\mathbf{g}^{\mathrm{active}}_l \sim \mathcal{N}(0, 
       \sigma \mathbf{I}_{\mathcal{L}^{\mathrm{ES}}_{\mathrm{active}}})$. \\
    2. Sample $\mathbf{g}^{\perp}_l \sim \mathcal{N}(0, 
       \sigma \mathbf{I}_{\mathcal{L}^{\mathrm{ES}, \perp}_{\mathrm{active}}})$. \\  
    3. Ask for $F(\theta_{t} \pm \mathbf{g}^{\mathrm{type}}_l)$ for $\mathrm{type} \in
      \{\mathrm{active}, \perp\}$. \\    
    4. Compute $v_{l}^{\mathrm{type}} = \frac{1}{2\sigma}(F(\theta_{t} +
       \mathbf{g}^{\mathrm{type}}_l)-F(\theta_{t}-\mathbf{g}^{\mathrm{type}}_l))$. \\     
    5. Compute $s^{\mathrm{active}}_l = \frac{l-1}{l}* s^{\mathrm{active}}_{l-1} +
       \frac{(v^{\mathrm{active}}_l)^2}{l}$. \\
    6. Compute $s^{\perp}_l = \frac{l-1}{l}* s^{\perp}_{l-1} +
       \frac{(v^{\mathrm{\perp}}_l)^2}{l}$.       
}
\textbf{Return:} $\widehat{r} = \sqrt{\frac{s^{\mathrm{active}}_C}{s^{\perp}_C}}$.
\caption{Explore estimator via compressed sensing}
\label{Alg:uniform_sampling}
\end{algorithm}

\subsection{Estimating the sensing ratio $r$.}
In this section we provide guarantees for the estimation of the ratio $r$ as specified in Section \ref{sec:exploration} for
Algorithm \ref{Alg:uniform_sampling}. 
Recall the definitions $s_{\mathbf{U}^{\mathrm{act}}} = \| \mathbf{U}^\top \nabla F(\theta_t)  \|^2$ and $s_{\mathbf{U}^{\perp}}  = \|  (\mathbf{U}^\perp)^\top \nabla F(\theta_t)  \|^2 $.

Since $\left| \frac{ F(\theta_t + \sigma \mathbf{g}) - F(\theta_t - \sigma \mathbf{g})}{2\sigma} - \mathbf{g}^\top \nabla F(\theta_t) \right| \leq \xi_\mathbf{g}(\theta_t)$, when $\mathbf{g} \sim \widehat{P}$, we recognize two cases. If $\mathbf{g} \sim \mathcal{N}(0, \mathbf{I}_{\mathbf{U}})$  the distribution of $\frac{F(\theta_t + \sigma \mathbf{g}) - F(\theta_t - \sigma \mathbf{g})}{2\sigma} \approx N(0, \| \mathbf{U}^\top \nabla F(\theta_t) \|^2 )$. Analogously when  $\mathbf{g} \sim \mathcal{N}(0, \mathbf{I}_{\mathbf{U}^\perp})$ the distribution of $\frac{F(\theta_t + \sigma \mathbf{g}) - F(\theta_t - \sigma \mathbf{g})}{2\sigma} \approx N(0, \| (\mathbf{U}^\perp)^\top \nabla F(\theta_t) \|^2 )$.

\begin{theorem}
Let $0<s<C$ and
let $\mathbf{g}_i \sim \mathcal{N}(0, \mathbf{I}_{\mathcal{L}^{\mathrm{ES}}_{(\mathrm{active})}})$
for $i=1,...,s$ and $\mathbf{g}_{i}, \sim \mathcal{N}(0, \mathbf{I}_{\mathcal{L}_{\mathrm{active}}^{\mathrm{ES}, \perp}})$ for $i=s+1,...,C$. 
Let $\hat{s}_{\mathbf{U}^{\mathrm{ort}}} := \frac{1}{s}\sum_{j=1}^s \left(  \frac{F(\theta + \sigma \mathbf{g}_j) - F(\theta - \sigma \mathbf{g}_j) }{2\sigma} \right)^2$, $\hat{s}_{\mathbf{U}^{\perp}}  
:=\frac{1}{C-s} \sum_{j=1}^{C-s} \left( \frac{F(\theta + \sigma \mathbf{g}_j) - F(\theta - \sigma \mathbf{g}_j) }{2\sigma}\right)^2$ and let $\hat{r} = 
\sqrt{\frac{\hat{s}_{\mathbf{U}^{\mathrm{ort}}} }{\hat{s}_{\mathbf{U}^{\perp}}}} $. Given $u, \epsilon > 0$ and $\delta  \in (\epsilon,1)$, the following holds.
\begin{enumerate}
\item If $C = 2s$ for $s \geq \frac{16}{u^2}\log\left(\frac{8}{\delta}\right)$ and under the mechanism from Algorithm \ref{Alg:uniform_sampling} or
\item If $\{ \mathbf{g}_i \}_{i=1}^C$ are samples generated under $\widehat{P}$,  $\min(p^{t}, 1-p^{t}) > u$ and $C \geq \max\left(  \frac{8}{(p^{t}-u) u^2 }  ,   
\frac{8}{(1-p^{t}-u)u^2}, \frac{2p^{t} + 2u/3}{u^2}  \right) \log\left(  \frac{12}{\delta}     \right)$, 
\end{enumerate}
then with probability at least $1-\delta$:
\begin{equation*}
    \sqrt{\frac{s_{\mathbf{U}^{\mathrm{ort}}}(1-u) - \frac{2\epsilon}{\delta} }{ s_{\mathbf{U}^{\perp}}(1+u) + \frac{2\epsilon}{\delta}  } }  \leq \widehat{r} \leq \sqrt{\frac{s_{\mathbf{U}^{\mathrm{act}}}(1+u) + \frac{2\epsilon}{\delta} }{ s_{\mathbf{U}^{\perp}}(1-u) - \frac{2\epsilon}{\delta}}}.
\end{equation*}
\end{theorem}

 \begin{proof}
 First observe we introduce some notation.
 \begin{equation*}
     \xi_\mathbf{g}^{(2)}(\theta_t) =  \left(  \frac{F(\theta_t + \sigma \mathbf{g}) - F(\theta_t - \sigma \mathbf{g}) }{2\sigma}   \right)^2 - \left(  \mathbf{g}^\top \nabla F(\theta_t)  \right)^2 
 \end{equation*}
 Observe that:
\begin{align*} 
 \xi^{(2)}_{\mathbf{g}}(\theta_t) &= \left| \left(  \frac{F(\theta_t + \sigma \mathbf{g}) - F(\theta_t - \sigma \mathbf{g}) }{2\sigma}   \right)^2 - \left(  \mathbf{g}^\top \nabla F(\theta_t)  \right)^2 \right| \\
 &\leq \xi_\mathbf{g}(\theta_t)^2 + 2 \left|  \mathbf{g}^\top \nabla F(\theta_t) \xi_\mathbf{g}(\theta_t)  \right|\\
 &\leq \sigma^4 \tau^2 \| \mathbf{g} \|^6 + 2 \sigma^2 \tau L \|\mathbf{g} \|^4
 \end{align*}

 Let $\hat{s}_V = \hat{s}_V^0 + \frac{1}{s} \sum_{j=1}^{s} \xi_{\mathbf{g}_j}(\theta_t)$ and $\hat{s}_{V^\perp} = \hat{s}_{V^\perp}^0 + \frac{1}{C-s} \sum_{j=1}^{C-s} \xi_{\mathbf{g}_j}(\theta_t)$. Where $s_{\mathbf{U}^{\mathrm{act}}}^0 = \frac{1}{s} \sum_{j=1}^s \left( \nabla F(\theta_t) ^\top \mathbf{g}_j   \right)^2$ and $s_{\mathbf{U}^{\perp}}^0 = \frac{1}{C-s} \sum_{j=s}^C \left( \nabla F(\theta_t) ^\top \mathbf{g}_j   \right)^2$.
 
 Notice that $\nabla F(\theta_t)^\top \mathbf{g}$ is distributed as a Gaussian Random variable (with variance depending on the support of the covariance of $\mathbf{g}$). 

 By concentration of squared gaussian random variables:
        \begin{align*}
            \mathbb{P}\left[ | \hat{s}_{V}^0 -  s_{\mathbf{U}^{\mathrm{act}}} | \geq u s_{\mathbf{U}^{\mathrm{act}}}     \right] \leq 2\exp\left(-\frac{s u^2}{8}\right)  \\
            \mathbb{P}\left[     | \hat{s}_{V^\perp}^0 -  s_{\mathbf{U}^{\perp}} | \geq u s_{\mathbf{U}^{\perp}}   \right] \leq 2\exp\left( -\frac{(k-s)u^2}{8} \right)
        \end{align*}
Consequently, with probability $1-2\exp\left(  -\frac{su^2}{8}  \right) - 2\exp\left(  -\frac{(k-s)u^2}{8} \right)$, it holds that:
\begin{equation*}
    \frac{s_{\mathbf{U}^{\mathrm{act}}}}{s_{\mathbf{U}^{\perp}}}\left( \frac{1+u}{1-u}    \right) \geq \frac{\hat{s}_V^0}{\hat{s}_{V^\perp}^0} \geq \frac{s_{\mathbf{U}^{\mathrm{act}}}}{s_{\mathbf{U}^{\perp}}} \left( \frac{1-u }{1+u} \right)
\end{equation*}

Notice that by Markov's inequality:

\begin{equation}
   \mathbb{P}\left( \xi_\mathbf{g}(\theta_t) \geq \frac{2\epsilon}{ \delta}\right) \leq \delta \frac{\mathbb{E}\left[   \sigma^4 \tau^2 \| \mathbf{g} \|^6 + 2 \sigma^2 \tau L \|\mathbf{g} \|^4  \right]}{2\epsilon} \leq \frac{\delta}{2}
\end{equation} 

Since $ \sigma < \frac{1}{35}\sqrt{\frac{\epsilon \min(p^{t}, 1-p^{t})}{\tau  d^3 \max(L, 1)} }$, $\mathbb{E}\left[   \sigma^4 \tau^2 \| \mathbf{g} \|^6 + 2 \sigma^2 \tau L \|\mathbf{g} \|^4  \right] \leq \epsilon$. 

Regardless if $\mathbf{g}$ was sampled from $\mathcal{N}(0, \mathbf{I})$, $\mathcal{N}(0, \mathbf{I}_{\mathbf{U}})$ or $\mathcal{N}(0, \mathbf{I}_{\mathbf{U}^\perp})$.

\textbf{Case 1.}

By definition, $C = 2s$, and therefore $C-s = C/2$ and therefore:
        \begin{align*}
            \mathbb{P}\left[ | \hat{s}_{V}^0 -  s_{\mathbf{U}^{\mathrm{act}}} | \geq u s_{\mathbf{U}^{\mathrm{act}}}     \right] \leq 2\exp\left(-\frac{C u^2}{16}\right)  \\
            \mathbb{P}\left[     | \hat{s}_{V^\perp}^0 -  s_{\mathbf{U}^{\perp}} | \geq u s_{\mathbf{U}^{\perp}}   \right] \leq 2\exp\left( -\frac{Cu^2}{16} \right)
        \end{align*}
We require that:
\begin{align*}
    2\exp\left(-\frac{C u^2}{16}\right) \leq  \delta/4 \\
    2\exp\left( -\frac{Cu^2}{16} \right) \leq \delta/4
\end{align*}

\textbf{Case 2.}

    In fact, by concentration results on Bernoulli variables, given $\alpha >0$, $| s - kp^{t}  |\leq k\alpha$ and $| (k-s) - (1-p^{t})k | \leq k\alpha$ with probability at least $1 - \exp\left( -\frac{k\alpha^2}{2p^{t}} \right) - \exp\left(   -\frac{k\alpha^2 }{ 2p^{t} + 2\alpha/3 }\right)$. 
    
    Let $\alpha = u$. Conditioning on the events that $|s-kp^{t}| \leq u k$ and $|(k-s) - (1-p^{t})k| \leq u k$. We seek to ensure that:
    
    \begin{align*}
        2\exp\left( -\frac{ (p^{t}-u)k u^2}{8}   \right) \leq \delta /6 \\
        2\exp\left(  - \frac{(1-p^{t}-u)k u^2}{8}      \right) \leq  \delta/ 6\\
        \exp\left( -\frac{ku^2}{2p^{t}}    \right) \leq \delta/6   \\
        \exp\left(-\frac{k u^2}{2p^{t} + 2u/3}      \right) \leq \delta /6
    \end{align*}

\textbf{Case 1 and 2}
    The following inequalities hold:

\begin{equation*}
    \frac{s_{\mathbf{U}^{\mathrm{act}}}(1-u) - \frac{2\epsilon}{\delta} }{ s_{\mathbf{U}^{\perp}}(1+u) + \frac{2\epsilon}{\delta}  } \leq \frac{\hat{s}_V^0 - \frac{2\epsilon}{\delta}} {  \hat{s}_{V^\perp}^0 + \frac{2\epsilon}{\delta}  } \leq \frac{\hat{s}_V }{\hat{s}_{V^\perp}  } \leq \frac{\hat{s}_V^0 + \frac{2\epsilon}{\delta}} {  \hat{s}_{V^\perp}^0 - \frac{2\epsilon}{\delta}  } \leq \frac{s_{\mathbf{U}^{\mathrm{act}}}(1+u) + \frac{2\epsilon}{\delta} }{ s_{\mathbf{U}^{\perp}}(1-u) - \frac{2\epsilon}{\delta}  }
\end{equation*}
    
    The union bound yields the desired result. And therefore the result follows.
\end{proof}

\end{document}